\documentclass[12pt]{amsart}

\usepackage{amsfonts,amssymb,stmaryrd,amscd,amsmath,latexsym,amsbsy}

\usepackage{amssymb}
\usepackage{amsfonts}
\usepackage{latexsym}

\newtheorem{theorem}{Theorem}[section]

\newtheorem{proposition}[theorem]{Proposition}
\newtheorem{corollary}[theorem]{Corollary}
\theoremstyle{definition}
\newtheorem{definition}[theorem]{Definition}

\newtheorem{example}[theorem]{Example}
\newtheorem{question}[theorem]{Question}
\newtheorem{problem}[theorem]{Problem}

\newtheorem{remark}[theorem]{Remark}

\newcommand{\Tr}{\text{Tr}}
\newcommand{\id}{\text{id}}

\newcommand{\End}{\text{End}}
\newcommand{\Hom}{\text{Hom}}

\newcommand{\Rep}{\text{Rep}}

\newcommand{\g}{\mathfrak{g}}
\newcommand{\kk}{\mathfrak{k}}
\newcommand{\p}{\mathfrak{p}}

\newcommand{\n}{\mathfrak{n}}
\renewcommand{\u}{\mathfrak{u}}
\renewcommand{\l}{\mathfrak{l}}
\newcommand{\z}{\mathfrak{z}}

\newcommand{\ben}{\begin{enumerate}}
\newcommand{\een}{\end{enumerate}}

\theoremstyle{plain}

\newtheorem*{sol}{Solution}

\theoremstyle{definition}

\theoremstyle{remark}

\newcommand{\solu}[1]{\begin{sol}{\bf (\ref{#1})}}

\begin{document}

\title{Representation theory in complex rank, II}

\author{Pavel Etingof}
\address{Department of Mathematics, Massachusetts Institute of Technology,
Cambridge, MA 02139, USA}
\email{etingof@math.mit.edu}
\maketitle

\vskip .1in
\centerline{\bf To the memory of Andrei Zelevinsky}
\vskip .1in

\section{Introduction}

Let ${\mathcal C}_t$ be one of the categories
$\Rep(S_t)$, $\Rep(GL_t)$, $\Rep(O_t)$, $\Rep(Sp_{2t})$,
obtained by interpolating the classical representation categories 
${\bf Rep}(S_n)$, ${\bf Rep}(GL_n)$, ${\bf Rep}(O_n)$, ${\bf Rep}(Sp_{2n})$
to complex values of $n$, defined by Deligne-Milne and 
Deligne (\cite{DM,De1,De2}).\footnote{Note that, to avoid confusion, we denote ordinary representation categories by 
${\bf Rep}$, and interpolated ones by $\Rep$.} In \cite{E1}, by analogy with 
the representation theory of real reductive groups, we proposed to consider 
various categories of ``Harish-Chandra modules'' based on ${\mathcal C}_t$, 
whose objects $M$ are objects of ${\mathcal C}_t$ equipped with 
additional morphisms satisfying certain relations. In this situation, 
the structure of an object of ${\mathcal C}_t$ on $M$ 
is analogous to the action of the ``maximal compact subgroup'', 
while the additional morphisms play the role of the 
``noncompact part''. The papers \cite{E1, EA, Ma} study examples of such categories 
based on the category $\Rep(S_t)$ 
(which could be viewed as doing ``algebraic combinatorics in 
complex rank''). This paper is a sequel to these works, and its goal is to 
start developing ``Lie theory in complex rank'', extending 
the constructions of \cite{E1} to ``Lie-theoretic'' 
categories $\Rep(GL_t)$, $\Rep(O_t)$, $\Rep(Sp_{2t})$ (based on the ideas outlined in \cite{E2}). 
Namely, we define complex rank analogs of 
the parabolic category O and the representation categories of real reductive Lie groups and supergroups, 
affine Lie algebras, and Yangians. 
We develop a framework and language for studying these categories, 
prove basic results about them, and outline a number of directions of 
further research. We plan to pursue these directions in future papers. 
 
The organization of the paper is as follows. 

In Section 2, we give some preliminaries on groups in tensor categories, 
and then give background on Deligne categories $\Rep(GL_t)$, $\Rep(O_t)$,
$\Rep(Sp_{2t})$. In particular, we show that the groups $GL_t$, $O_t$, $Sp_{2t}$ are connected, 
and therefore can be effectively studied by looking at their Lie algebras.  

In Section 3, we explain how to interpolate classical symmetric pairs,  
and then proceed to discuss the basic theory of Harish-Chandra modules. 

In Section 4, we develop the interpolation of 
the basic theory of parabolic category O.

In Section 5, we discuss the interpolations of classical Lie supergroups, 
$GL_{t|s}$ and $OSp_{t|2s}$. 

In Section 6, we discuss the interpolation of the representation theory of affine Lie algebras.  

Finally, in Section 7, we describe the interpolation of Yangians of classical Lie algebras. 
 
{\bf Acknowledgments.} The author is grateful to 
I. Entova-Aizenbud and V. Ostrik for many useful discussions. 
The work of the author was  partially supported by the NSF grants
DMS-0504847 and DMS-1000113.

\section{Preliminaries}

\subsection{Symmetric tensor categories}

Throughout the paper, we work over the field $\Bbb C$ of complex numbers. 
By a symmetric tensor category, we will mean a $\Bbb C$-linear 
artinian\footnote{An artinian category is an abelian category whose objects have finite length and morphism spaces are finite dimensional.} 
rigid symmetric monoidal category ${\mathcal{C}}$, with biadditive tensor product and ${\rm End}(\bold 1)=\Bbb C$, as in \cite{DM}
(such categories are also called {\it pre-Tannakian categories}, or {\it tensor categories satisfying finiteness assumptions} (\cite{De1}, 2.12.1)). 
We will not keep track of bracket positions in tensor products of several objects. 
For $V\in {\mathcal{C}}$, we will denote by ${\rm ev}_V : V^*\otimes V\to \bold 1$ and ${\rm coev}_V: \bold 1\to V\otimes V^*$
the evaluation and coevaluation morphisms of $V$. 

Recall that in a symmetric tensor category ${\mathcal{C}}$, it makes sense to talk about any linear algebraic structure 
(such as a (commutative) associative algebra, Lie algebra, module over such an algebra, etc). 
We will also routinely consider ordinary algebraic structures (over $\Bbb C$) as those in ${\mathcal{C}}$,
by using the functor $A\mapsto A\otimes \bold 1$.   
If ${\mathcal{D}}$ is an artinian category, by ${\rm Ind}{\mathcal{D}}$ we will mean the ind-completion of ${\mathcal{D}}$; 
it consists of inductive limits of objects of ${\mathcal{D}}$ (for instance, the ind-completion of the category of finite dimensional vector spaces is the category of all vector spaces). 

\subsection{Affine group schemes in symmetric tensor categories and their Lie algebras}

Recall that an affine group scheme $G$ in a symmetric tensor category 
${\mathcal{C}}$ corresponds to a commutative Hopf algebra $H$ in ${\mathcal{C}}$; we write $H=O(G)$ and $G={\rm Spec}H$. 
To such a group scheme $G$ we can attach the category ${\bf Rep}(G)$ of representations of $G$ in ${\mathcal{C}}$, which is, by definition, the category of 
(left) $O(G)$-comodules in ${\mathcal{C}}$. 

Note that $O(G)$ carries two commuting actions of $G$ 
preserving the algebra structure -- left translations and right translations; as coactions, they are both defined by the coproduct $\Delta: O(G)\to O(G)\otimes O(G)$. 
The corresponding diagonal action is called the {\it adjoint action}, and it preserves the Hopf algebra structure. 

For an affine group scheme $G$ we can define its Lie algebra 
in the same way as in classical Lie theory. 
Namely, let $I$ be the augmentation 
ideal in $H$. Then $I/I^2$ has a natural Lie coalgebra structure
(this is a general property of Hopf algebras), so $\g:=(I/I^2)^*$ (which is, in general, 
a pro-object) is a Lie algebra, denoted by $\g={\rm Lie}G$. 

Moreover, if $M$ is a locally algebraic $G$-module (i.e., a left $H$-comodule in ${\rm Ind}{\mathcal{C}}$) then we have a natural 
map $\zeta_M: M\to I/I^2\otimes M$ (the categorical analog of the map $x\mapsto \rho(x)-1\otimes x$ taken modulo $I^2$ in the 
first component, where $\rho: M\to H\otimes M$ is the coaction) which defines an action $\g\otimes M\to M$ 
of $\g$ on $M$ (the categorical analog of the derivative of a Lie group representation). 
This gives rise the ``derivative'' functor $D: {\rm Ind}{\bf Rep}(G)\to {\rm Ind}{\bf Rep}(\g)$. 
In particular, taking $M=H$ and $\rho=\Delta$ (the coproduct), we get an action $\delta: \g\otimes H\to H$ 
of $\g$ on $H$ by algebra derivations (this is a categorical analog of infinitesimal left translations).   

\begin{definition} We say that an affine group scheme $G$ in ${\mathcal{C}}$ is connected 
if $H^\g:={\rm Ker}\zeta_H$ is isomorphic to $\bold 1$.\footnote{For classical algebraic groups $G$ over $\Bbb C$, this definition coincides 
with the usual definition of connectedness: it says that a regular function on $G$ annihilated by all the right-invariant vector fields on $G$ must be constant.} 
\end{definition}

In particular, we see that the connectedness property 
is preserved under symmetric tensor functors. 

\begin{proposition}\label{ff} 
$G$ is connected if and only if the functor 
$$
D: {\rm Ind}{\bf Rep}(G)\to {\rm Ind}{\bf Rep}(\g)
$$  
is fully faithful. 
\end{proposition}

\begin{proof}
Suppose $G$ is connected. We need to show that for any $X,Y\in {\bf Rep}(G)$, any $\g$-homomorphism 
$f: X\to Y$ is actually a $G$-homomorphism. For this, it's enough to show that 
for any $G$-module $U$ in ${\mathcal{C}}$, one has $U^G=U^\g$; indeed, then we can take $U=X^*\otimes Y$ 
(note that a priori, we only know that $U^G\subset U^\g\subset \Hom_{\mathcal{C}}(\bold 1,U)$). 
We have an inclusion $U\to O(G)\otimes U_{\rm obj}$, where $U_{\rm obj}$ is the underlying object of $U$ with the trivial $G$-action 
(the coaction of $O(G)$ on $U$). Thus, it suffices to check that 
$O(G)^G=O(G)^\g$ (i.e., the invariants under the usual and infinitesimal left translations coincide), i.e., that $O(G)^\g=\bold 1$, which is the definition of connectedness. 

Conversely, if the functor $D$ is fully faithful then $O(G)^G=O(G)^\g$, so $O(G)^\g=\bold 1$, i.e., $G$ is connected.  
\end{proof} 

\subsection{Classical groups in symmetric tensor categories} 

Let us now define the general linear, orthogonal, and symplectic groups 
in symmetric tensor categories. 

\begin{definition}\label{clgroups}
(i) Let $V$ be an object in a symmetric tensor category ${\mathcal{C}}$. 
The group scheme $GL(V)$ is cut out inside $V\otimes V^*\oplus V^*\otimes V$ 
by the equations $AB=BA={\rm Id}$ (i.e., $O(GL(V))$ is the quotient of $S(V^*\otimes V\oplus V\otimes V^*)$
by the ideal $J$ defined by these equations). 

(ii) Suppose that $V$ is equipped with a symmetric (respectively, skew-symmetric) 
isomorphism $\psi: V\to V^*$. 
The group $O(V)$ (respectively, $Sp(V)$) is cut out inside 
$V\otimes V^*$ by the equations $AA^*=A^*A={\rm Id}$, where $A^*$ is 
the adjoint of $A$ with respect to $\psi$. 
\end{definition}

The structure of an affine group scheme on $GL(V)$, $O(V)$, $Sp(V)$ is defined in the same way as in classical 
Lie theory. 

\begin{remark} 1. The equations in Definition \ref{clgroups} are to be understood categorically; 
e.g., $AB=BA={\rm Id}$ means that the ideal $J$ is generated by  
the images of $V^*\otimes V$ inside $S(V^*\otimes V\oplus V\otimes V^*)$
under the morphisms 
$$
\sigma_{3,4}\circ({\rm Id}_{V^*}\otimes {\rm coev}_V\otimes {\rm Id}_V)-{\rm ev}_V\text{ and }
\sigma_{3,4}\sigma_{12,34}\circ ({\rm Id}_{V^*}\otimes {\rm coev}_V\otimes {\rm Id}_{V})-{\rm ev}_V
$$   
(where $\sigma$ denotes the permutation of the appropriate components). 

2. In classical Lie theory one of the two equations $AB=BA={\rm Id}$ or $AA^*=A^*A={\rm Id}$ suffices (and implies the other), but  
we don't expect that this is the case in general. The proof of this implication uses determinants, 
which are not available in general, and the statement fails for infinite dimensional 
vector spaces (which don't form a rigid category, however). 
\end{remark}

Now let us describe the Lie algebras of the groups $GL(V)$, $O(V)$, $Sp(V)$. 
First of all, for any $V$, the object ${\mathfrak{gl}}(V):=V\otimes V^*$ 
is naturally an associative algebra and thus a Lie algebra, with the bracket being the commutator. 
Next, if $V$ is equipped with a symmetric (respectively, skew-symmetric) isomorphism 
$\psi: V\to V^*$, then we can define an automorphism of Lie algebras 
$\theta: {\mathfrak{gl}}(V)\to {\mathfrak{gl}}(V)$ given by 
$$
\theta=-\sigma(\psi\otimes \psi^{-1}),
$$ 
and one can define the Lie algebra ${\mathfrak{o}}(V)$ (respectively, ${\mathfrak{sp}}(V)$) to be ${\rm Ker}(\theta-{\rm Id})$. 
Note that as objects we have ${\mathfrak{o}}(V)=\wedge^2V$ and ${\mathfrak{sp}}(V)=S^2V$. 

\begin{proposition}\label{leaal}
We have ${\rm Lie}GL(V)={\mathfrak{gl}}(V)$, ${\rm Lie}O(V)={\mathfrak{o}}(V)$, 
${\rm Lie}Sp(V)={\mathfrak{sp}}(V)$. 
\end{proposition} 

\begin{proof}
This is readily obtained as in classical Lie theory, by linearizing the equations defining the corresponding groups. 
\end{proof} 

Observe that we have a Lie subalgebra ${\mathfrak{sl}}(V)\subset {\mathfrak{gl}}(V)$, where ${\mathfrak{sl}}(V)$ 
is the kernel of the evaluation morphism. This Lie algebra is the Lie algebra of the group scheme 
$PGL(V)=GL(V)/\Bbb C^*$, defined by the equality $O(PGL(V))=O(GL(V))_0$, where the subscript $0$ means the degree zero part 
under the $\Bbb Z$-grading on $O(GL(V))$ in which $A$ has degree $1$ and $B$ has degree $-1$. 
Note, however, that in general, we cannot define the group scheme $SL(V)$ (as the determinant 
character of $GL(V)$ is not defined).

\subsection{The fundamental group of a symmetric tensor category}

Let us recall the basic theory of fundamental groups of symmetric tensor categories (\cite{De1}, Section 8). 
If ${\mathcal{C}}$ is a symmetric tensor category, then one can define 
a commutative algebra $R_{\mathcal{C}}$ in ${\rm Ind}({\mathcal{C}}\boxtimes {\mathcal{C}})$ by the formula 
$$
R_{\mathcal{C}}:=(\oplus_{X\in {\mathcal{C}}}X\boxtimes X^*)/E,
$$
where $E$ is the sum of the images of the morphisms 
$$
f\boxtimes {\rm Id}_{Y^*}-{\rm Id}_X\boxtimes f^*
$$
over all objects $X,Y\in {\mathcal{C}}$ and morphisms $f: X\to Y$. 
The multiplication in $R_{\mathcal{C}}$ 
is just the tensor product (i.e., it tautologically maps 
$(X\boxtimes X^*)\otimes (Y\boxtimes Y^*)$ to $(X\otimes Y)\boxtimes (Y^*\otimes X^*)$). 
If ${\mathcal{C}}$ is semisimple, then $R_{\mathcal{C}}=\oplus_X X\boxtimes X^*$,
where $X$ runs over the isomorphism classes of simple objects of ${\mathcal{C}}$. 

Let $H_{\mathcal{C}}=T(R_{\mathcal{C}})\in {\rm Ind}{\mathcal{C}}$, 
where $T: {\mathcal{C}}\boxtimes {\mathcal{C}}\to {\mathcal{C}}$ is the tensor product functor 
(so $H_{\mathcal{C}}=\oplus_X X\otimes X^*$ in the semisimple case). 
Then $H_{\mathcal{C}}$ is a commutative Hopf algebra. Indeed, the coproduct maps 
$X\otimes X^*$ to $(X\otimes X^*)\otimes (X\otimes X^*)$ by means of 
the morphism ${\rm Id}_X\otimes {\rm coev}_{X^*}\otimes {\rm Id}_{X^*}$. 
This Hopf algebra can be viewed as the algebra $O(\pi({\mathcal{C}}))$ of regular functions 
on an affine group scheme $\pi({\mathcal{C}})$ in ${\mathcal{C}}$, which is called the {\it fundamental group} 
of ${\mathcal{C}}$. Note that every object $X\in {\mathcal{C}}$ has a natural action of $\pi({\mathcal{C}})$ 
(i.e., a coaction of $H_{\mathcal{C}}$). 

If $F: {\mathcal{C}}\to {\mathcal{D}}$ is a symmetric tensor functor between two
symmetric tensor categories, then we have a natural homomorphism of Hopf algebras 
$\xi_F: F(H_{\mathcal{C}})\to H_{\mathcal{D}}$.
Consider the category ${\bf Rep}_{\mathcal{D}}(\pi({\mathcal{C}}))$
of representations of $\pi({\mathcal{C}})$ in ${\mathcal{D}}$, which by definition is 
the category of $Y\in {\mathcal{D}}$ with a coaction $\tau: Y\to Y\otimes F(H_{\mathcal{C}})$  
such that $({\rm Id}_Y\otimes \xi_F)\circ \tau: Y\to Y\otimes H_{\mathcal{D}}$ is the canonical coaction of $H_{\mathcal{D}}$ on $Y$. 
 
The following theorem comes out of the standard formalism of fundamental groups: 

\begin{theorem} (\cite{De1}, Theorem 8.17)
The functor $F$ defines an equivalence of categories 
${\mathcal C}\to {\bf Rep}_{\mathcal{D}}(\pi({\mathcal{C}}))$.  
\end{theorem}   

\subsection{The category $\Rep(GL_t)$} 

Let us review the definition and known results about $\Rep(GL_t)$. 

The category $\Rep(GL_t)$ was first defined in \cite{DM}, Examples 1.26, 1.27
(see also \cite{De1,De2} and \cite{CW} for a review). It is obtained 
by interpolating ${\bf Rep}(GL_n)$ to non-integer values of $n$, as follows.  

Recall that in the classical category ${\bf Rep}(GL_n)$ we have the 
vector representation $V=\Bbb C^n$, and every irreducible representation 
of $GL_n$ occurs in $V^{\otimes r}\otimes V^{*\otimes s}$ for some $r$, $s$. 
Now, 
$$
\Hom(V^{\otimes r_1}\otimes V^{*\otimes s_1},V^{\otimes r_2}\otimes V^{*\otimes s_2})=
\Hom(V^{\otimes r_1+s_2},V^{\otimes r_2+s_1}),
$$
so it is nonzero only if $r_1+s_2=r_2+s_1=m$, and 
in the latter case is spanned by elements 
of $\Bbb C[S_m]$, by the Fundamental Theorem of Invariant Theory
(this spanning set is a basis if $n\ge m$). 
The category ${\bf Rep}(GL_n)$ can then be defined 
as the (additive) Karoubian closure of the subcategory 
with objects $[r,s]:=V^{\otimes r}\otimes V^{*\otimes s}$ 
and morphisms as above. 

Now consider composition of morphisms. 
To do so, note that the elements of $S_m$ defining 
morphisms can be depicted as oriented planar tangles (with possibly intersecting strands) 
with $r_1$ inputs and $s_1$ outputs on the bottom and 
$r_2$ inputs and $s_2$ outputs on the top, and 
$m$ arrows, each going from an input to an output. 
The composition of morphisms is then defined 
as concatenation of tangles, followed by closed loop removal, 
with each removed loop earning a factor of $n$. 
For example, if $A: [1,1]\to [1,1]$ 
is given by $A={\rm coev}_V\circ {\rm ev}_V$,
then $A^2=nA$.  

Now, given $t\in \Bbb C$, one can define the category $\widetilde{\Rep}(GL_t)$ 
with objects $[r,s]$, $r,s\in \Bbb Z_+$, 
and the space of morphisms 
$\Hom([r_1,s_1],[r_2,s_2])$ 
being spanned by planar tangles as above, with the same composition law 
as above, except that every removed closed loop earns a factor of $t$. 

\begin{remark} The endomorphism algebra $\End([r,s])$ is called the {\it walled Brauer algebra} and denoted $B_{r,s}(t)$. 
\end{remark}

Note that the category $\widetilde{\Rep}(GL_t)$ has a natural 
strict symmetric monoidal structure. Namely, the tensor product functor is just the addition of pairs of integers
for objects and taking the union of planar tangles for morphisms, with the obvious symmetric braiding.
The unit object is the object $[0,0]$.  

\begin{definition} 
The category $\Rep(GL_t)$ 
is the Karoubian closure of $\widetilde{\Rep}(GL_t)$
(i.e., it is obtained from $\widetilde{\Rep}(GL_t)$ 
by adding the images of all the idempotent morphisms). 
\end{definition} 

Clearly, $\Rep(GL_t)$ is a Karoubian category 
(i.e.,  an idempotent-closed additive category) over $\Bbb C$, which inherits 
the tensor structure from $\widetilde{\Rep}(GL_t)$. Moreover, it is not hard to show that 
this category is rigid (with $[r,s]^*=[s,r]$). 
Moreover, it is easy check that $\dim [r,s]=t^{r+s}$;
this is just the interpolation of the equality $\dim(V^{\otimes r}\otimes V^{*\otimes s})=n^{r+s}$.   

\begin{theorem} (\cite{DM,De1,De2}) (i) 
The category $\Rep(GL_t)$ is a semisimple abelian symmetric 
tensor category if $t\notin \Bbb Z$.

(ii) The category $\Rep(GL_t)$ has the following universal property: 
if ${\mathcal C}$ is a rigid tensor category then 
isomorphism classes of (possibly non-faithful) symmetric tensor functors $\Rep(GL_t)\to {\mathcal C}$ 
are in bijection with isomorphism classes of objects $X\in {\mathcal C}$ of dimension $t$, via
$F\mapsto F([1,0])$. 

(iii) If $t=n\in \Bbb Z$, and if $p,q$ are nonnegative integers with $p-q=n$, 
then the category $\Rep(GL_{t=n})$ (which is not abelian)   
admits a non-faithful symmetric tensor functor $\Rep(GL_n)\to {\bf Rep}(GL_{p|q})$ to the representation category 
of the supergroup $GL_{p|q}$, which sends $[1,0]$ to the supervector space $V=\Bbb C^{p|q}$. 

(iv) We have a natural symmetric tensor functor ${\rm Res}: \Rep(GL_t)\to \Rep(GL_{t-1})$. 
\end{theorem}

Note that (iii) and (iv) are easy consequences of (ii). 

Let's consider the case $t\notin \Bbb Z$. 
In this case, simple objects in $\Rep(GL_t)$ are labeled by pairs of arbitrary partitions, $(\lambda,\mu)$, 
$\lambda=(\lambda_1,...,\lambda_r)$, $\mu=(\mu_1,...,\mu_s)$. 
Namely, letting $V=[1,0]$ be the tautological object 
(the interpolation of the defining representation), 
we have simple objects $X_{\lambda,\mu}$ 
which are direct summands in $S^\lambda V\otimes S^\mu V^*$,
where $S^\lambda$ is the Schur functor corresponding to the partition 
$\lambda$. More specifically, $X_{\lambda,\mu}$ is the only direct summand 
in $S^\lambda V\otimes S^\mu V^*$ 
which does not occur in $S^{\lambda'}V\otimes S^{\mu'}V^*$ 
with $|\lambda'|<|\lambda|$. This summand occurs with multiplicity $1$. 
All of this is readily seen by noting that this is the case 
in ${\bf Rep}(GL_n)$ for large $n$, in which case 
$X_{\lambda,\mu}$ is the irreducible representation 
$V_{[\lambda,\mu]_n}$ of $GL_n$, with highest weight $[\lambda,\mu]_n$, where 
$$
[\lambda,\mu]_n=(\lambda_1,...,\lambda_r,0,...,0,-\mu_s,...,-\mu_1)
$$
(here, the string of zeros in the middle has length $n-r-s$).

Thus, we should think of $X_{\lambda,\mu}$ as the interpolation of 
the representation $V_{[\lambda,\mu]_n}$ to complex values of $n$; in particular, $X_{\lambda,\mu}^*=X_{\mu,\lambda}$.  
Consequently, the dimension of $X_{\lambda,\mu}$ is given by the interpolation of the 
Weyl dimension formula: 
\begin{equation}\label{dimfor}
\dim X_{\lambda,\mu}(t)=\\
d_\lambda d_\mu \prod_{i=1}^r\frac{\binom{t+\lambda_i-i-s}{\lambda_i}}{\binom{\lambda_i+r-i}{\lambda_i}}
\prod_{j=1}^s\frac{\binom{t+\mu_j-j-r}{\mu_j}}{\binom{\mu_j+s-j}{\mu_j}}
\prod_{i=1}^r\prod_{j=1}^s\frac{t+1+\lambda_i+\mu_j-i-j}{t+1-i-j}, 
\end{equation}
where 
$$
d_\lambda=\dim V_\lambda=\prod_{1\le i<j\le r}\frac{\lambda_i-\lambda_j+j-i}{j-i}
$$
is the dimension of 
the irreducible representation of $GL_{|\lambda|}$ with highest weight $\lambda$. 
Note that since this function takes integer values at large positive integer $t$, 
it is an integer-valued polynomial (a linear combination of binomial coefficients $\binom{t}{k}$). 

\subsection{The categories $\Rep(O_t)$ and $\Rep(Sp_{2t})$.}

The category $\Rep(O_t)$ is defined similarly to the category $\Rep(GL_t)$. 
Namely, recall that in the classical category ${\bf Rep}(O_n)$ we have the 
vector representation $V=\Bbb C^n$, and every irreducible representation 
of $O_n$ occurs in $V^{\otimes r}$ for some $r$. 
Now, 
$$
\Hom(V^{\otimes r_1},V^{\otimes r_2})=
(V^{\otimes r_1+r_2})^{O_n},
$$
so it is nonzero only if $r_1+r_2=2m$, in which case it can be written as $\End_{O_n}(V^{\otimes m})$ and 
is the image of the Brauer algebra $B_m(n)$, by the Fundamental Theorem of Invariant Theory
for orthogonal groups (this image is isomorphic to the Brauer algebra if $n\ge m$). 
The category ${\bf Rep}(O_n)$ can then be defined 
as the Karoubian closure of the subcategory 
with objects $[r]:=V^{\otimes r}$ 
and morphisms as above. 

Now consider composition of morphisms. 
A basis in the Brauer algebra $B_m(n)$ 
is formed by matchings of $2m$ points, 
so we have a spanning set in $\Hom(V^{\otimes r_1},V^{\otimes r_2})$
formed by unoriented planar tangles (with possibly intersecting strands) connecting $r_1$ points at the bottom and $r_2$ points at
the top, which define a perfect matching. Then composition is the concatenation of tangles, 
followed by removal of closed loops, so that each removed loop is replaced by a factor of $n$.  

Now, given $t\in \Bbb C$, one can define the category $\widetilde{\Rep}(O_t)$ 
with objects $[r]$, $r\in \Bbb Z_+$, 
and the space of morphisms 
$\Hom([r_1],[r_2])$ 
being spanned by planar tangles as above, with the same composition law 
as above, except that every removed closed loop earns a factor of $t$. 
Thus, for instance, the endomorphism algebra $\End([m])$ is the 
Brauer algebra $B_m(t)$. 

Similarly to $\widetilde{\Rep}(GL_t)$, the category $\widetilde{\Rep}(O_t)$ has a natural 
strict symmetric monoidal structure. Namely, the tensor product functor is just the addition of integers
for objects and taking the union of planar tangles for morphisms, with the obvious symmetric braiding.
The unit object is the object $[0]$.  

\begin{definition} 
The category $\Rep(O_t)$ 
is the Karoubian closure of $\widetilde{\Rep}(O_t)$.
\end{definition} 

Clearly, $\Rep(O_t)$ is a Karoubian category over $\Bbb C$, which inherits 
the tensor structure from $\widetilde{\Rep}(O_t)$. Moreover, it is not hard to show that 
this category is rigid (with $[r]^*=[r]$). 
Moreover, it is easy check that $\dim [r]=t^{r}$.   

The category $\Rep(Sp_{2t})$ is defined in a completely parallel way, 
starting from the representation category of the symplectic group $Sp_{2n}$. 
It is in fact easy to see that the categories $\Rep(O_t)$ and 
$\Rep(Sp_{-t})$ are equivalent as tensor categories, 
and differ only by a change of the commutativity isomorphism. 
Namely, define an involutive tensor automorphism $u$ of the identity functor 
of $\Rep(O_t)$ (called the parity automorphism) by $u|_{[r]}=(-1)^r$, 
and define a new commutativity isomorphism on $\Rep(O_t)$ 
which differs by sign from the old one if both factors are odd (i.e., $u=-1$ on them), 
and is the same as the old one if one of the factors is even (i.e., has $u=1$). Then it is easy to see that 
$\Rep(O_t)$ with this new commutativity is equivalent to $\Rep(Sp_{-t})$ 
as a symmetric tensor category.\footnote{There is a similar relationship between the categories $\Rep(GL_t)$ and $\Rep(GL_{-t})$.}  

\begin{theorem} (\cite{De1,De2}) (i) 
The category $\Rep(O_t)$ is a semisimple abelian symmetric 
tensor category if $t\notin \Bbb Z$.

(ii) The category $\Rep(O_t)$ (respectively, $\Rep(Sp_t)$) has the following universal property: 
if ${\mathcal C}$ is a rigid tensor category then isomorphism classes of (possibly non-faithful) 
symmetric tensor functors $\Rep(O_t)\to {\mathcal C}$ (respectively $\Rep(Sp_t)\to {\mathcal C}$)  
are in bijection with isomorphism classes of objects $X\in {\mathcal C}$ of dimension $t$ with a symmetric (respectively, skew-symmetric) isomorphism $X\to X^*$, via
$F\mapsto F([1])$. 

(iii) If $t=n\in \Bbb Z$, and if $p,q$ are nonnegative integers with $p-2q=n$, 
then the category $\Rep(O_{t=n})$ (which is not abelian)   
admits a non-faithful symmetric tensor functor $\Rep(O_n)\to {\bf Rep}(OSp_{p|2q})$ to the representation category 
of the supergroup $OSp_{p|2q}$, 
which sends $[1]$ to the supervector space $V=\Bbb C^{p|2q}$. 

(iv) We have a natural symmetric 
tensor functor ${\rm Res:} \Rep(O_t)\to \Rep(O_{t-1})$ and 
$\Rep(Sp_{2t})\to \Rep(Sp_{2t-2})$. 
\end{theorem}

Again, (iii) and (iv) follow from (ii). 

Now assume $t\notin \Bbb Z$ and let us describe the simple objects. 
The simple objects $X_\lambda$ of $\Rep(O_t)$ are labelled by all partitions 
$\lambda=(\lambda_1,...,\lambda_r)$; namely, $X_\lambda$ is the 
unique direct summand in $S^\lambda V$ which does not occur in $S^{\lambda'}V$ for any $\lambda'$ with 
$|\lambda'|<|\lambda|$ (it occurs with multiplicity $1$). The object $X_\lambda$ is the interpolation 
of the representation $V_\lambda$ of $O_n$ with highest weight $\sum \lambda_i\omega_i$, where 
$\omega_i$ are the fundamental weights corresponding to the representation 
$\wedge^i V$. 

Thus, the dimension of $X_\lambda$ is given by the interpolation of the Weyl dimension formula: 
$$
\dim X_\lambda(t)=
$$
$$
\prod_{i=1}^r\frac{ (\frac{t}{2}+\lambda_i-i) \binom{\lambda_i+t-r-i-1}{\lambda_i} } { (\frac{t}{2}-i) \binom{\lambda_i+r-i}{\lambda_i} }
\prod_{1\le i<j\le r}\frac{(\lambda_i-\lambda_j+j-i)(\lambda_i+\lambda_j+t-i-j)}{(j-i)(t-i-j)}.
$$
Note that since this function takes integer values at large positive integer $t$, 
it is an integer-valued polynomial.

We will refer to $\Rep(GL_t)$, $\Rep(O_t)$, $\Rep(Sp_t)$ as {\it Deligne categories}. 
In this paper we will consider these categories only in the semisimple case $t\notin \Bbb Z$,
but many of our constructions can be extended to the general case.  

\subsection{Tensor subcategories} 

Proper tensor subcategories of the Deligne categories are easy to classify, since 
they are seen at the level of the Grothendieck ring. 

The category $\Rep(GL_t)$ is $\Bbb Z$-graded (by $\deg X_{\lambda,\mu}=|\lambda|-|\mu|$). 
So for every positive integer $N$ we have the subcategory $\Rep(GL_t/\Bbb Z_N)$ of representations of degrees 
divisible by $N$, and the subcategory $\Rep(PGL_t)$ of representations of degree zero. 

The categories $\Rep(O_t)$ and $\Rep(Sp_{2t})$ are $\Bbb Z_2$-graded, by $\deg(V)=1$, so 
we have the subcategories $\Rep(O_t/(\pm 1))$ and $\Rep(Sp_{2t}/(\pm 1))$
of even representations. 

It is easy to check that these are the only nontrivial tensor subcategories
of the Deligne categories. 

\subsection{The fundamental groups of Deligne categories}

Denote the fundamental groups of $\Rep(GL_t)$, $\Rep(PGL_t)$, $\Rep(O_t)$, $\Rep(Sp_{2t})$ 
by $GL_t$, $PGL_t$, $O_t$, $Sp_{2t}$, respectively. The following proposition provides 
an explicit description of these fundamental groups.

Recall that $V$ denotes the tautological object of the Deligne category. 

\begin{proposition}\label{expdes}
(i) $GL_t=GL(V)$, $PGL_t=PGL(V)$. 

(ii) $O_t=O(V)$, and $Sp_{2t}=Sp(V)$.   
\end{proposition} 

\begin{proof}
Since $\Rep(GL_t)$ is tensor-generated\footnote{A tensor category ${\mathcal{C}}$ is said to be tensor-generated by 
objects $X_1,...,X_m$ if any object of ${\mathcal{C}}$ is a subquotient of a direct sum of objects of the form 
$X_{i_1}\otimes...\otimes X_{i_n}$, $1\le i_1,...,i_n\le m$.} 
 by $V$ and $V^*$ 
and $\Rep(O_t)$, $\Rep(Sp_{2t})$ are tensor-generated by 
$V$, we find that $GL_t$ is a closed subscheme of $V\otimes V^*\oplus V^*\otimes V$, 
while $O_t$ and $Sp_{2t}$ are closed subschemes of $V\otimes V^*$. 
It's easy to see that the defining equations are satisfied on each of these subschemes, and 
one can check that they are sufficient (i.e., modulo these equations 
one already obtains the Hopf algebra $H=\oplus_X X\otimes X^*$).  
\end{proof} 

\begin{example}
Let us explain how this works in the example of $GL_t$. 
Let us work in $\Rep(GL_t)\boxtimes \Rep(GL_t)$. 
In this case, the algebra in question is 
$$
R:=S(V^*\boxtimes V\oplus V\boxtimes V^*)/(AB=BA={\rm Id}),
$$
where $V^*\boxtimes V$ corresponds to ``matrix elements of $A$'' and $V\boxtimes V^*$ to 
``matrix elements of $B$''. This algebra has a filtration by degree in $A$ and $B$. In degree $0$, we have just $\bold 1$. 
In degree $1$, we additionally have $V^*\boxtimes V$ and $V\boxtimes V^*$ corresponding to $A$ and $B$, respectively. 
In degree $2$, before imposing the relations, we additionally have 
$S^2(V^*\boxtimes V)\oplus (V^*\boxtimes V)\otimes (V\boxtimes V^*)\oplus S^2(V\boxtimes V^*)$. 
Note that $S^2(V^*\boxtimes V)=S^2V^*\boxtimes S^2V\oplus \wedge^2V^*\boxtimes \wedge^2V$. Now, the two relations 
$AB-{\rm Id}=0$ and $BA-{\rm Id}=0$ kill the two subobjects 
$\bold 1\boxtimes (V\otimes V^*)$ and $(V^*\otimes V)\boxtimes \bold 1$  
in $(V^*\boxtimes V)\otimes (V\boxtimes V^*)$ (intersecting by $\bold 1$, as $\Tr(AB)=\Tr(BA)$), 
which leaves us with ${\mathfrak{sl}}(V)\boxtimes {\mathfrak{sl}}(V)^*$. 
Thus, the additional summands in degree 2 are: 
$$
S^2V^*\boxtimes S^2V\oplus \wedge^2V^*\boxtimes \wedge^2V\oplus S^2V\boxtimes S^2V^*\oplus \wedge^2V\boxtimes \wedge^2V^*\oplus 
{\mathfrak{sl}}(V)\boxtimes {\mathfrak{sl}}(V)^*. 
$$
Similarly, one can show that in higher degrees $d>2$ we get one copy of $X\boxtimes X^*$ for each simple $X$ which occurs in $V^{\otimes r}\otimes V^{*\otimes s}$ 
with $r+s\le d$.   
\end{example} 

\begin{corollary} \label{liealg} 
We have ${\rm Lie}GL_t={\mathfrak{gl}}(V)$, ${\rm Lie}PGL_t={\mathfrak{sl}}(V)$, ${\rm Lie}O_t={\mathfrak{o}}(V)$, 
${\rm Lie}Sp_{2t}={\mathfrak{sp}}(V)$. 
\end{corollary} 

\begin{proof} This follows from Proposition \ref{expdes} and Proposition \ref{leaal}.  
\end{proof} 

We will denote these Lie algebras by ${\mathfrak{gl}}_t$, ${\mathfrak{sl}}_t$, ${\mathfrak{o}}_t$, 
${\mathfrak{sp}}_{2t}$. As we have shown above, they act naturally (i.e., functorially with respect to $M$) 
on every (ind-)object $M$ of the corresponding Deligne category. 

\subsection{Connectedness of $GL_t$, $PGL_t$, $O_t$, $Sp_{2t}$}

Let us denote any of the group schemes $GL_t$, $PGL_t$, $O_t$, $Sp_{2t}$ by $K$ and the corresponding Lie algebra by $\kk$. 

\begin{proposition}\label{conne} 
The group scheme $K$ is connected. 
\end{proposition} 

\begin{remark}
Note that if $n$ is a positive integer then the group $O_n$ is not connected. However, this is due to the 
existence of the determinant character for $O_n$, which does not exist for $O_t$.  
\end{remark} 

\begin{proof}
Let $X$ be a nontrivial simple object of $\Rep(K)$. Then it is easy to see that $\kk$ acts nontrivially on $X$.
This implies that $O(K)^\kk=\bold 1$, as desired. 
\end{proof} 

\begin{remark}
On the contrary, it is easy to show that the group scheme $S_t$ in $\Rep(S_t)$ (defined in \cite{De2}) is "totally disconnected" 
in the sense that ${\rm Lie}(S_t)=0$. 
\end{remark} 

\section{Interpolation of the representation theory of real
reductive groups}

\subsection{Interpolation of classical symmetric pairs} 

Let $(\g,\kk)$ be a symmetric pair, i.e. $\g$ is a complex
reductive Lie algebra, and $\kk$ the fixed subalgebra 
of an involution $\theta: \g\to \g$. Let $K$ be a reductive group
whose Lie algebra is $\kk$. The main algebraic objects of study 
in the representation theory of real reductive groups 
are $(\g,K)$-modules. These, by definition, are locally algebraic $K$-modules with a
compatible action of $\g$. The category of such modules will be denoted by $\Rep(\g,K)$. 

We would like to define the interpolation of the category 
$\Rep(\g,K)$ to complex rank in the case when the Lie algebra $\g$ is 
of classical type. To do so, let us give a ``categorically
friendly'' formulation of the additional structure on a locally
algebraic $K$-module $M$ that gives it a compatible $\g$-action. 

To this end, note that $\g=\kk\oplus \p$, where $\p$ is the
$-1$-eigenspace of $\theta$ (which is a $K$-module), and we have a bracket map 
$$
\eta: \wedge^2\p\to \kk,
$$
which is a morphism of $K$-modules. 
Then, a structure of a $(\g,K)$-module on a $K$-module $M$ is
just a morphism of $K$-modules 
$$
b: \p\otimes M\to M
$$
such that 
\begin{equation} \label{commrel}
b\circ ({\rm Id}_\p \otimes b)=a_M\circ (\eta\otimes {\rm Id}_M),
\end{equation}
as morphisms $\wedge^2\p\otimes M\to M$
(where on the left hand side we regard $\wedge^2\p$ as a subobject of $\p\otimes \p$). 

So for each symmetric pair with classical $\g$ (and hence $\kk$) we can define 
the interpolation of its category of $(\g,K)$-modules as the 
category of ind-objects $M$ of the Deligne category interpolating the category ${\bf Rep}(K)$
with a morphism $b: \p\otimes M\to M$ satisfying
(\ref{commrel}). The only thing we have to do for this is to
define the appropriate object $\p$ with the morphism $\eta$. 

Let us explain how this works in examples, following the classification of symmetric spaces (\cite{He}). 

\begin{example}\label{gt} Group type (the symmetric pair $(K\times K,K)$). 
This example works in any of the Deligne categories 
$\Rep(GL_t)$, $\Rep(O_t)$, $\Rep(Sp_{2t})$.
Namely, $\p=\kk$, and the map $\eta: \wedge^2\p\to \kk$ 
is the commutator. In other words, in this case a $(\g,K)$-module 
is simply a $\kk$-module in the Deligne category 
(i.e., an (ind-)object $M$ of the Deligne category with a $\kk$-action, which does not necessarily 
coincide with the natural action of $\kk$ on $M$). 
We denote the category of such modules by  
$\Rep(\kk_t\times \kk_t, K_t)$ for the corresponding $\kk=\kk_t={\mathfrak{gl}}_t,{\mathfrak{o}}_t,{\mathfrak{sp}}_t$. 
\end{example}

\begin{example}\label{AI} Type AI (the symmetric pair $(GL_n,O_n)$). 
The appropriate Deligne category is 
$\Rep(O_t)$, and $\p=S^2V$, with the bracket 
$$
\eta: \wedge^2\p=\wedge^2(S^2V)\to \kk=\wedge^2V
$$
given by the formula
$$
\eta={\rm Id}_V\otimes (,)\otimes {\rm Id}_V.
$$
We denote the resulting category of $(\g,K)$-modules by 
$\Rep({\mathfrak{gl}}_t,O_t)$. 
\end{example}

\begin{example}\label{AII} Type AII (the symmetric pair $(GL_{2n},Sp_{2n})$). 
The appropriate Deligne category is 
$\Rep(Sp_{2t})$, and $\p=\wedge^2V$, with the bracket 
$$
\eta: \wedge^2\p=\wedge^2(\wedge^2V)\to \kk=S^2V
$$
given by the same formula as in Example \ref{AI}. 
We denote the resulting category of $(\g,K)$-modules by 
$\Rep({\mathfrak{gl}}_{2t},Sp_{2t})$. 
\end{example}

\begin{example}\label{AIII} Type AIII (the symmetric pair $(GL_{n+m},GL_n\times
GL_m$)). 
The appropriate Deligne category is 
$\Rep(GL_t)\boxtimes \Rep(GL_s)$ 
(so we have two complex parameters).
Let $V$ and $U$ be the tautological objects 
of these two categories. 
Then $\p=V\otimes U^*\oplus U\otimes V^*$, with the bracket 
$$
\eta: \wedge^2\p\to \kk=V\otimes V^*\oplus
U\otimes U^*
$$
given by the formula
$$
\eta=({\rm Id}_V\otimes {\rm ev}_U\otimes {\rm Id}_{V^*})\circ (p_1\otimes p_2)
-({\rm Id}_U\otimes {\rm ev}_V\otimes {\rm Id}_{U^*})\circ (p_2\otimes p_1),
$$
where $p_1,p_2$ are the projections to the first and
second summand of $\p$. We denote the resulting category by 
$\Rep({\mathfrak{gl}}_{t+s},GL_t\times GL_s)$. 
\end{example}

\begin{example}\label{BDI} Type BDI (the symmetric pair ($O_{n+m},O_n\times
O_m$)). The appropriate Deligne category is 
$\Rep(O_t)\boxtimes \Rep(O_s)$. 
Let $V$ and $U$ be the tautological objects 
of these two categories. 
Then $\p=V\otimes U$, with the bracket 
$$
\eta: \wedge^2\p\to \kk=\wedge^2 V\oplus
\wedge^2 U
$$
given by the formula
$$
\eta=({\rm Id}_V\otimes (,)_U\otimes {\rm Id}_V)\circ \sigma_{34} 
-({\rm Id}_U\otimes (,)_V\otimes {\rm Id}_U)\circ \sigma_{12}
$$
We denote the resulting category by 
$\Rep({\mathfrak o}_{t+s},O_t\times O_s)$. 
\end{example}

\begin{example}\label{CII} Type CII (the symmetric pair $(Sp_{2(n+m)},Sp_{2n}\times
Sp_{2m})$). The appropriate Deligne category is 
$\Rep(Sp_{2t})\boxtimes \Rep(Sp_{2s})$. 
Let $V$ and $U$ be the tautological objects 
of these two categories. 
Then $\p=V\otimes U$, with the bracket 
$$
\eta: \wedge^2\p\to \kk=S^2 V\oplus
S^2 U
$$
given by the same formula as in Example \ref{BDI}.
We denote the resulting category by 
$\Rep({\mathfrak{sp}}_{2(t+s)},Sp_{2t}\times Sp_{2s})$. 
\end{example}

\begin{remark}
Note that in the last three examples, 
one can freeze one of the parameters ($t$ or $s$) to 
be a positive integer (i.e., use the usual representation category of the corresponding Lie group, rather than the Deligne category), 
and consider the interpolation only with
respect to the other parameter.
\end{remark}

\begin{example} Type DIII (the symmetric pair $(O_{2n},GL_n)$). 
The appropriate Deligne category is 
$\Rep(GL_t)$, and $\p=\wedge^2V\oplus \wedge^2V^*$, 
with the bracket 
$$
\eta: \wedge^2\p=\wedge^2(\wedge^2V\oplus \wedge^2V^*)\to
\kk=V\otimes V^*
$$
given by the formula
$$
\eta=({\id}_V\otimes (,)\otimes {\rm Id}_{V^*})\circ P,
$$
where $P: \wedge^2(\wedge^2V\oplus \wedge^2V^*)\to \wedge^2 V\otimes \wedge^2V^*$ is the projection. 
We denote the resulting category by 
$\Rep({\mathfrak o}_{2t},GL_t)$. 
\end{example}

\begin{example} Type CI (the symmetric pair $(Sp_{2n},GL_n)$). 
The appropriate Deligne category is 
$\Rep(GL_t)$, and $\p=S^2V\oplus S^2V^*$, 
with the bracket 
$$
\eta: \wedge^2\p=\wedge^2(S^2V\oplus S^2V^*)\to
\kk=V\otimes V^*
$$
given by the formula
$$
\eta=({\id}_V\otimes (,)\otimes {\rm Id}_{V^*})\circ P,
$$
where $P: \wedge^2(S^2V\oplus S^2V^*)\to S^2 V\otimes S^2V^*$ is the projection. 
We denote the resulting category by 
$\Rep({\mathfrak{sp}}_{2t},GL_t)$. 
\end{example}

Note that all the above complex rank categories $\Rep(\g,K)$ can be defined using a slightly different language. 
Namely, we have a Lie algebra $\g=\kk\oplus \p$ in $\Rep(K)$, whose commutator 
is composed of the usual commutator on $\kk$, the action of $\kk$ on $\p$, and the map $\eta$, 
and $\Rep(\g,K)$ is nothing but the category of $\g$-modules $M$ in ${\rm Ind}\Rep(K)$, such that 
the restriction of the $\g$-action on $M$ to $\kk$ coincides with the natural action of $\kk$ on $M$. 

For example, in the group type case (Example \ref{gt}), we have $\g=\kk\oplus \kk$, and $\kk$ 
sits in $\g$ as the diagonal subalgebra. Thus, the objects of $\Rep(\kk\oplus \kk,K)$
can be viewed as $\kk$-bimodules in $\Rep(K)$ such that the diagonal $\kk$-action is the natural one. 

In fact, the above definition becomes more natural in light of the following construction, 
which also provides examples of finite dimensional $(\g,K)$-modules.

\begin{example}\label{fdhc} 
1. Consider the setting of Example \ref{gt}. It is easy to see that we have a symmetric tensor functor 
$F: \Rep(K)\boxtimes \Rep(K)\to \Rep(\kk\oplus \kk,K)$ given by $X\boxtimes Y\mapsto X\otimes Y$.
The additional action of $\kk$ on $X\otimes Y$ is just the action on the left component, while the natural action of $\kk$ is the 
diagonal one. Thus, the two actions coincide iff $Y$ is a multiple of $\bold 1$. 

2. In the non-group type examples above, by the universal property of Deligne categories, 
we have a symmetric tensor functor $F: \Rep(G)\to \Rep(\g,K)$, where $\Rep(G)$ 
is the Deligne category corresponding to $\g=\g_t$ in each of the cases (i.e. $G=GL_t$ 
if $\g={\mathfrak{gl}}_t$, etc.) 
 \end{example} 

\begin{proposition} \label{ff1}
The functor $F$ of Example \ref{fdhc} is fully faithful. 
\end{proposition} 

\begin{proof}
This follows from Propositions \ref{ff} and \ref{conne}. 
\end{proof} 

\subsection{The center of $U(\g)$.}

A fundamental role in the representation theory of real reductive groups is played by the center $Z(\g)$ of the enveloping algebra $U(\g)$. 
So, let us discuss the structure of this center in our setting. 

Note that $U(\g)$ is a bimodule over the ordinary algebra $U(\g)^\kk=\Hom(\bold 1,U(\g))$. 
By definition, $Z(\g)$ is the subalgebra of $U(\g)^\kk$ 
consisting of elements 
$z$ whose left and right action on $U(\g)$ are the same. 
In particular, $Z(\g)$ is an ordinary commutative algebra (over $\Bbb C$). 
Since we are in characteristic zero, we have a symmetrization map $S\g\to U(\g)$, which is a map of $\g$-modules. 
Hence, the center of $U(\g)$ is identified, as a vector space, with $(S\g)^\g$, and thus we have that ${\rm gr}Z(\g)=
(S\g)^\g$ (where we take the associated graded algebra under the usual PBW filtration on the enveloping algebra).   

\begin{proposition}\label{cent} (i) The algebra $(S\kk)^\kk=\Hom(\bold 1,S\kk)$
 is a polynomial algebra in infinitely many homogeneous generators $C_i$,
 of  degrees $i=1,2,3,...$ if $\kk={\mathfrak{gl}}_t$ 
and degrees $i=2,4,6,...$ if $\kk={\mathfrak{o}}_t$ or ${\mathfrak{sp}}_{2t}$.
 So in Example \ref{gt}, $(S\g)^\g$ is a polynomial algebra in two strings of such generators,
 $C_i^{\rm left}$ and $C_i^{\rm right}$. 
 
(ii) In all the other examples, 
the algebra $(S\g)^\g$ is a polynomial algebra 
in infinitely many homogeneous generators $C_i$ of degrees $i=1,2,3,...$ if $\g$ is of type ${\mathfrak{gl}}$ 
and degrees $i=2,4,6,...$ if $\g$ is of type ${\mathfrak{o}}$ or ${\mathfrak{sp}}$. 

(iii) The center $Z(\g)$ is a polynomial algebra, whose generators are obtained from the generators of $(S\g)^\g$ by the symmetrization map. 
\end{proposition}

\begin{proof}
(i) Since $K$ is connected by Proposition \ref{conne}, $(S\kk)^\kk=(S\kk)^K$, so this is just a calculation of invariants in 
the Deligne category $\Rep(K)$. Thus, the statement follows from classical invariant theory by looking at large integer $t$. 
Namely, identifying $\g$ with $\g^*$, the generators of $(S\g)^\g=(S\g^*)^\g$ may be written as $C_i=\Tr(A^i)$. 

(ii) By Proposition \ref{ff1}, the algebra $(S\g)^\g$ may be computed in the Deligne category $\Rep(G)$ as 
the algebra $(S\g)^G$. Thus, (ii) follows from (i). 

(iii) follows from (i),(ii) and the fact that ${\rm gr}Z(\g)=(S\g)^\g$.  
\end{proof} 

In fact, we can generalize Proposition \ref{cent}(i) to the algebra $((S\g)^{\otimes m})^\g$. 

\begin{proposition}\label{multiinv}
(i) If $G=GL_t$ then $((S\g)^{\otimes m})^\g$ is the polynomial algebra in the generators $C_w$ 
labeled by cyclic words (=necklaces) $w$ in letters $A_1,...,A_m$ (namely, 
$C_w$ is the interpolation of $\Tr(w)$). The degree of $C_w$ 
is the length of $w$. 

(ii) If $G=O_t$ or $Sp_{2t}$ then $((S\g)^{\otimes m})^\g$ is the polynomial algebra in the generators $C_w$ 
labeled by cyclic words $w$ in letters $A_1,...,A_m$ modulo reversal, except palindromic words $w$ of odd length (namely, 
$C_w$ is the interpolation of $\Tr(w)$). The degree of $C_w$ 
is the length of $w$.
\end{proposition} 

\begin{proof} It suffices to check this in $\Rep(G)$, where it follows from the invariant theory for classical groups. 
Namely, to settle the $GL_t$-case, recall that by Weyl's Fundamental Theorem of Invariant Theory, 
the ring of invariants of $m$ square matrices $A_1,...,A_m$ is generated by traces of cyclic words of these matrices, 
and these traces are asymptotically independent when the matrix size goes to infinity (see e.g. \cite{CEG}, Section 11). 
For $O_t$ and $Sp_{2t}$, the proof is similar; namely, one needs to use the well known fact that there is no polynomial identities 
satisfied by skewsymmetric matrices (under an orthogonal or symplectic form) of arbitrary size. 
\end{proof} 

Using standard combinatorics (necklace counting), we get 

\begin{corollary}\label{hilser} 
For $G=GL_t$ the Hilbert series of $((S\g)^{\otimes m})^\g$ is 
$$
h_m(q)=\prod_{j=1}^\infty (1-mq^j)^{-1}. 
$$
\end{corollary} 

\subsection{Kostant's theorem} 

Now we would like to generalize the results of Kostant \cite{Ko} 
to Deligne categories. 

\begin{proposition}\label{kos} 
$S\g$ is a free module over $(S\g)^\g$. More precisely, there exists a $\g$-stable graded subobject 
$E\subset S\g$ such that the multiplication map $E\otimes (S\g)^\g\to S\g$ is an isomorphism.
\end{proposition} 

\begin{proof} It suffices to prove the result in $\Rep(G)$. It is sufficient to show 
that for each simple $X\in \Rep(G)$, the space $\Hom_\g(X,S\g)$ is a free module over $(S\g)^\g$. 
This follows from the fact that $(S\g\otimes S\g)^\g$ is a free $(S\g)^\g$-module, which 
is a consequence of Proposition \ref{multiinv} (for $m=2$).  
\end{proof} 

In fact, similarly to the classical case, there is a nice choice for $E$ (at least for generic $t$). 
Namely, we can define the harmonic part $H(\g)\subset S\g$, which 
by definition is the kernel of the action of the positive degree elements $(S\g)_+^\g\subset 
(S\g)^\g$ by constant coefficient differential operators (using the identification $\g\cong \g^*$).
In other words, one has $H(\g)=((S\g)_+^\g S\g)^\perp$, where the orthogonal complement is taken under 
the natural nondegenerate form on $S\g$ (the interpolation of the form 
defined in the classical case by the formula $(f,g)=f(\partial)g(x)|_{x=0}$). 
We have the multiplication map $\mu: H(\g)\otimes (S\g)^\g \to S\g$. 

\begin{proposition}\label{kos1} If $t$ is transcendental, then the map $\mu$ is an isomorphism. 
In other words, in Proposition \ref{kos}, one can choose $E=H(\g)$. 
Moreover, the Hilbert series of $E$ and $H(\g)$ are the same for any $t\notin \Bbb Z$. 
\end{proposition} 

\begin{proof} The first statement follows from its validity for large integer $t$
(for classical representation categories), which is a classical result of Kostant \cite{Ko}.
To prove the second statement, note that, as explained above, we have a perfect pairing $H(\g)\otimes S\g/(S\g)^\g_+S\g\to \bold 1$, which implies 
that $H(\g)\cong (S\g/(S\g)^\g_+S\g)^\ast$ as graded objects, where by $\ast$ we denote the restricted dual. 
Since by Proposition \ref{kos}, $S\g$ is freely generated by $E$ over $(S\g)^\g$, the result follows.   
\end{proof} 

Computing the Hilbert series of isotypic components of $E$ (or, equivalently, $H(\g)$) is an interesting problem. 
This is the complex rank analog of computing Kostant's generalized exponents of representations (=$q$-analogs 
of the zero weight multiplicity), and it leads to stable limits of these $q$-weight multiplicities for classical groups, studied 
by R. Gupta, P. Hanlon and R. Stanley in 1980s (\cite{Gu1,Gu2,Ha,St}). Namely, for instance, for $G=GL_t$ we have the following result. 

\begin{proposition}\label{sta} (\cite{St}, Proposition 8.1) Let $\lambda,\mu$ be partitions such that $|\lambda|=|\mu|$. 
Then the Hilbert series of $\Hom(X_{\lambda,\mu},E)$ is given by the formula
$$
h_{\Hom(X_{\lambda,\mu},E)}(q)=(s_\lambda*s_\mu)(q,q^2,...),
$$
where $s_\lambda$ are the Schur polynomials, and $*$ denotes the Kronecker product
(corresponding to the tensor product of representations of the symmetric group). 
\end{proposition} 

\begin{example}\label{adjointrep} Let $\lambda=\mu=(1)$. Then $s_\lambda*s_\mu=s_1*s_1=s_1=\sum x_i$, so 
$$
h_{\Hom({\mathfrak{sl}}(V),E)}(z)=s_1(q,q^2,q^3,...)=q+q^2+q^3+...=\frac{q}{1-q}. 
$$
\end{example}

Note that Proposition \ref{sta} implies the following combinatorial identity 
(which is easy to obtain from \cite{St} by interpolation): 
$$
\sum_{\lambda,\mu: |\lambda|=|\mu|}(s_\lambda*s_\mu)(q,q^2,...)\dim X_{\lambda,\mu}(t)=\frac{1}{(1-q)^{t^2}}\prod_{j=1}^\infty (1-q^j), 
$$
where $\dim X_{\lambda,\mu}(t)$ is given by formula \eqref{dimfor}. 

Generalizations of Proposition \ref{sta} 
to $O_t$ and $Sp_{2t}$ can be found in \cite{Ha} (Theorem 5.21, Corollary 5.17). 

\subsection{Harish-Chandra modules} 

By analogy with the classical case, we make the following definition. 

\begin{definition} A $(\g,K)$-module $M$ is said to be a Harish-Chandra module if 
it is finitely generated as a $\g$-module (i.e., is a quotient of $U(\g)\otimes X$ for some 
object $X\in \Rep(G)$) and is finite under the action of $Z(\g)$ (i.e., has a finite filtration 
such that $Z(\g)$ acts by a scalar on the successive quotients). 
The category of Harish-Chandra modules is denoted by $HC(\g,K)$. 
\end{definition}

For instance, any finite dimensional $(\g,K)$-module (e.g, one coming from an object of $\Rep(G)$)
is automatically a Harish-Chandra module. 

\begin{remark} We will see below that the finite $K$-type condition, satisfied automatically in the classical case, 
does not always hold in the setting of Deligne categories.  
\end{remark} 

\subsection{Dual principal series Harish-Chandra bimodules} 

Let us now give examples of infinite-dimensional Harish-Chandra 
bimodules for $K=K_t$. Namely, let us construct 
the dual principal series modules. 

We start with the spherical case. 
For a general categorical symmetric pair, 
let us say that $M\in \Rep(\g,K)$ is {\it spherical} if it contains 
a copy of the unit object $\bold 1$ of $\Rep(K)$ ("the spherical vector") which generates $M$. 
Let $Z=U(\kk)^\kk$, $\chi$ be a character of $Z$, and 
$$
U_\chi:=U(\kk)/(z-\chi(z),z\in Z).
$$ 
It is easy to see that $U_\chi\in HC(\kk\oplus \kk,K)$.
Also, we see that ${\rm gr}U_\chi\cong H(\kk)$. 
We call $U_\chi$ the {\it dual spherical principal series Harish-Chandra bimodule} with central character $\chi$. 

Moreover, for generic $\chi,t$ (in a suitable sense), the module 
$U_\chi$ is irreducible. Indeed, for each simple $\Rep(K)$-subobject
$X\subset U_\chi$ there exists $m$ such that $U^{(m)}_\chi X U^{(m)}_\chi$ 
(where $U_\chi^{(m)}$ is the degree $m$ part of $U_\chi$ under the PBW filtration on $U(\kk)$) 
contains $\bold 1$ for large integer $t$ and generic $\chi$, which implies the statement.   

\begin{remark}
Since ${\rm gr}U_\chi\cong H(\kk)$, $U_\chi$ does not have a finite $K$-type 
(see Example \ref{adjointrep}). This shows that in general, we should not expect finite $K$-type 
for irreducible Harish-Chandra modules in Deligne categories (although, as we will see below, some interesting 
Harish-Chandra modules do have finite $K$-type). 
\end{remark}

\begin{proposition}\label{sphe}
Any irreducible spherical Harish-Chandra bimodule $M\in \Rep(\g,K)$ is a quotient of $U_\chi$ for some $\chi$. 
In particular, $M$ contains a unique copy of $\bold 1$. 
\end{proposition}

\begin{proof}
By Dixmier's version of Schur's lemma (in the categorical setting), 
the center $Z$ acts in $M$ by some character $\chi$. Since $M$ contains a copy of $\bold 1$, 
we have a nonzero morphism of bimodules 
$$
(U_\chi\otimes U(\kk))/(U_\chi\otimes U(\kk))\kk_{diag}=U_\chi
$$
Since $M$ is irreducible, this morphism is surjective, as desired.  
\end{proof} 

The following problem is therefore interesting. 

\begin{problem} Determine the set $\Sigma=\Sigma_\kk$
of central characters $\chi$ for which $U_\chi$ is a reducible bimodule, i.e., is not a simple algebra
in $\Rep(K)$. 
\end{problem} 

We have just seen that this set is not everything (at least for transcendental $t$), but it is also nonempty. 

Indeed, consider e.g. the case $K=GL_t$. Then if $\chi$ equals the central character $\chi_{\lambda,\mu}$ of the object $X_{\lambda,\mu}$ 
then $U_\chi$ is not simple, as it projects onto $X_{\lambda,\mu}\otimes X_{\lambda,\mu}^*$. 
 
Let us compute $\chi_{\lambda,\mu}$ explicitly. To do so, we should choose generators $C_i$ of 
$U(\kk)^\kk=Z(\kk)$. We have the Duflo isomorphism of algebras
$$
{\rm Duf}: S(\kk)^\kk\to Z(\kk),
$$ 
defined in the same way as in the classical case (\cite{Du}; see \cite{CR} for a review). 
Set $C_i={\rm Duf}(\Tr(A^i))$. Then $C_i|_{V_{\lambda,\mu}}$ will be the interpolation 
of $\sum_j ([\lambda,\mu]_n+\rho_n)_j^i$, where $\rho_n$ is the half-sum of positive roots, 
i.e., 
$$
\chi_{\lambda,\mu}(C_i)=
\sum_j \left(\left(\lambda_j+\frac{t+1}{2}-j\right)^i-\left(\frac{t+1}{2}-j\right)^i\right)+
$$
$$
\sum_j\left(\left(-\mu_j-\frac{t+1}{2}+j\right)^i-\left(-\frac{t+1}{2}+j\right)^i\right)+P_i(t),
$$
where $P_i(t)$ is the modified Bernoulli polynomial, defined for positive integer $t$ by the formula 
$$
P_i(n)=\sum_{k=1}^n \left(\frac{n+1}{2}-k\right)^i;
$$
it is derived from the exponential generating function 
$$
\sum_{i\ge 0}P_i(t)\frac{z^i}{i!}=\frac{\sinh(z t/2)}{\sinh(z/2)}.
$$

Thus, we have $\chi_{\lambda,\mu}\in \Sigma$. 

In fact, it is not hard to see that $\chi_{\lambda,\mu}\in \Sigma$ not just when $\lambda$ and $\mu$ are partitions, 
but actually for any complex values of $\lambda_j$ and $\mu_j$. For instance, 
consider the case when $\lambda=(\ell)$ and 
$\mu=0$, i.e., $X_{\lambda,\mu}=S^\ell V$. 
Then we have a surjective algebra map 
$$
\phi_\ell: U(\kk)\to S^\ell V\otimes S^\ell V^*,
$$ 
and $\phi_\ell |_Z=\chi_{\ell,0}$. 
Let $I_\ell$ be the kernel of this map. Then $I_\ell$ contains an $\ell$-independent subobject $Y_2$ of $U(\kk)$ sitting in filtration 
degree $2$, which at the graded level gives the "rank 1" equation $\wedge^2A=0$, $A\in \kk$ (in the categorical setting), see 
\cite{BJ} (for ${\mathfrak{sl}}(V)\subset \kk$ this relation interpolates the quantization of the minimal coadjoint orbit). 
For $\lambda\in \Bbb C$ let $\widetilde{I}_\lambda$ 
be the ideal $(Y_2)+(C_1-\lambda)$. Then $Q_\lambda:=U(\kk)/\widetilde{I}_\lambda$ is a spherical 
Harish-Chandra bimodule which is a quotient of $U_{\chi_{\lambda,0}}$. 

It is easy to check that as an object of $\Rep(K)$, $Q_\lambda$ has a decomposition 
$$
Q_\lambda=\oplus_{m\ge 0}X_{m,m}
$$
(where $X_{m,m}$ is the simple summand of $S^mV\otimes S^mV^*$ not occurring in 
$V^{\otimes j}\otimes V^{*\otimes j}$ for $j<m$). In particular, $Q_\lambda$ has finite $K$-type. 
Using this decomposition, it is not hard to 
check that $Q_\lambda$ is irreducible if $\lambda$ is not an integer. 
On the other hand, if $\lambda=\ell$ is a positive integer, then 
$Q_\lambda$ is a length $2$ module which can be included in the non-split exact sequence 
$$
0\to \overline{Q}_\ell\to Q_\ell\to S^\ell V\otimes S^\ell V^*\to 0,
$$
where $\overline{Q}_\ell=\oplus_{m\ge \ell+1}X_{m,m}$,
and $S^mV\otimes S^mV^*=\oplus_{m=0}^\ell X_{m,m}$
are irreducible composition factors. Note that the Harish-Chandra bimodule 
$\overline{Q}_\ell$ is not spherical, even though its left and right central characters coincide. 

Similarly, if the length of $\lambda$ is $r$, the length of $\mu$ is $s$, and $r+s=p$, then there is a subobject $Y_{p+1}$ of $U(\kk)$ sitting in degree $p+1$ quantizing 
the relation $\wedge^{p+1}A=0$ in $S\kk$ that is annihilated by the homomorphism 
$$
U(\kk)\to X_{\lambda,\mu}\otimes X_{\lambda,\mu}^*,
$$ 
and we can define the quotient of $U(\kk)$ by the ideal generated by $Y_{p+1}$ and the relations $C_i=\gamma_i$, $i=1,...,p$, which gives an $p$-parameter family 
of spherical Harish-Chandra bimodules that are nontrivial quotients of the corresponding $U_\chi$. These are interpolations to complex rank
of quantizations of coadjoint orbits of ${\mathfrak{gl}}_n$ consisting of matrices of rank $p$ with fixed eigenvalues. 

We obtain the following proposition. 

\begin{proposition}\label{pointsinsigma} 
For any complex $\lambda,\mu$, 
the character $\chi_{\lambda,\mu}$ 
 belongs to $\Sigma$. 
\end{proposition} 

It would be interesting to know if $\Sigma$ contains any other points than $\chi_{\lambda,\mu}$.

\subsection{Non-spherical Harish-Chandra bimodules} 

Many more Harish-Chandra bimodules can be obtained from dual spherical principal series 
by applying functors of tensoring with finite dimensional bimodules. 
Namely, to each finite dimensional Harish-Chandra bimodule $Y$, 
we can attach the functor $T_Y: \Rep(\g,K)\to \Rep(\g,K)$ given by 
$T_Y(M)=M\otimes Y$ (the usual tensor product of $\kk\oplus \kk$-modules). 
For an irreducible Harish-Chandra bimodule $M$, the bimodule $T_Y(M)$ typically won't be irreducible, 
but one can look at its quotients corresponding to particular central characters  
of the left and right action of $U(\kk)$ (which will be Harish-Chandra bimodules, but in general will not be spherical). 

In general, if $M$ is irreducible, we don't expect $T_Y(M)$ to have finite length. 
However, it is not hard to check, for instance, that $Q_\lambda\otimes Y$ has finite length.
For example, take $Y=V$ (the tautological object under the left action of $\kk$ with the trivial right action). 
We have $X_{m,m}\otimes V=X_{m+1,m}\oplus X_{(m,1),m}$ (the last summand is missing for $m=0$). 
So, as ind-objects of $\Rep(K)$, we  have 
$$
Q_\lambda\otimes V=Q_\lambda'\oplus Q_\lambda'',
$$ 
where $Q_\lambda'=\oplus_{m\ge 0}X_{m+1,m}$ and $Q_\lambda''=\oplus_{m\ge 1}X_{(m,1),m}$.  
Interpolating from positive integer $\lambda$, one can easily show that this is in fact a decomposition of Harish-Chandra bimodules, 
and the subbimodules $Q_\lambda'$ and $Q_\lambda''$ are the eigenobjects of the left action of the center. 
Moreover, one can check that $Q_\lambda',Q_\lambda''$ are irreducible for non-integer $\lambda$. 

More generally, given two central characters $\chi_1,\chi_2$ of $Z$, we can define the category $HC(\kk\oplus\kk,K)_{\chi_1,\chi_2}$ of Harish-Chandra 
bimodules in which the left action of $Z$ is via $\chi_1$ and the right action via $\chi_2$. Clearly, every 
irreducible Harish-Chandra bimodule belongs to one of such categories. The following question is interesting. 

\begin{question} Which of the categories $HC(\kk\oplus\kk,K)_{\chi_1,\chi_2}$ are nonzero?
\end{question} 

Note that if $M\in HC(\kk\oplus\kk,K)_{\chi_1,\chi_2}$ and $X\subset M$ is a simple object of $\Rep(K)$, then we have a nonzero morphism 
$$
N(\chi_1,\chi_2,X):=(U_{\chi_1}\otimes U_{\chi_2}^{\rm op})\otimes_{U(\kk)}X\to M
$$ 
(where $\kk$ is embedded diagonally), so 
we see that $HC(\kk\oplus \kk,K)_{\chi_1,\chi_2}$ is nonzero iff $N(\chi_1,\chi_2,X)\ne 0$ for some simple object $X\in \Rep(K)$. 

\subsection{Dual spherical principal series in the general case}

The above constructions can be generalized to the case of symmetric pairs which are not of group type. 
Indeed, let us construct dual spherical principal series modules. Namely, given a character 
$\chi$ of $Z=U(\g)^\g$, consider the tensor product 
${\mathcal{I}}(\chi)=U_\chi\otimes_{U(\kk)}\bold 1$, 
where $U_\chi=U(\g)/(z-\chi(z),z\in Z)$. Then ${\mathcal{I}}(\chi)\in HC(\g,K)$.
As in the group case, it is easy to show that any spherical irreducible Harish-Chandra module 
is a quotient of ${\mathcal{I}}(\chi)$ for a unique $\chi$. 

\begin{remark} 1. As in the classical case, the module ${\mathcal{I}}(\chi)$ may sometimes be zero. This happens whenever $\chi$ does not vanish on the ideal 
$J=Z\cap U(\g)\kk\subset Z$, which may occur in case AIII (the symmetric pair $({\mathfrak{gl}}_{t+s},GL_t\times GL_s)$, $t,s\in \Bbb C$) and also in cases BDI, CII when one of the 
two parameters $t,s$ is fixed to be an integer. 

2. It is explained in \cite{He2} that 
for classical symmetric pairs $(G,K)$, 
the map $Z(\g)\to D(G/K)^G$ from the center of $U(\g)$ 
to the algebra of invariant differential operators 
on $G/K$ is onto. This implies that in the classical case, 
$U_\chi \otimes_{U(\kk)}\Bbb C$ is the usual dual principal series 
Harish-Chandra module for $G$.  
\end{remark} 

This gives rise to the following problem. 

\begin{problem} Find the set of $\chi$ for which ${\mathcal{I}}(\chi)$ is reducible, and describe irreducible quotients of ${\mathcal{I}}(\chi)$. 
\end{problem} 

Also, more general Harish-Chandra modules may be obtained from ${\mathcal{I}}(\chi)$ and its quotients by tensoring with finite dimensional 
Harish-Chandra modules (coming from objects of $\Rep(G)$), and then taking quotients by various central characters. 
Given $Y\in \Rep(G)$ and a Harish-Chandra module $M$ with central character $\chi_1$, it is an interesting question for which 
central characters $\chi_2$ the module 
$$
(M\otimes Y)/(z-\chi_2(z),z\in Z)(M\otimes Y)
$$ 
is nonzero.  

\subsection{Holomorphic discrete series} 

In the special cases of Hermitian symmetric spaces, i.e. when $\p=\u_+\oplus \u_-$ (namely, cases AIII, DIII,CI), 
we can define the subcategory $HC_+(\g,K)$ of $HC(\g,K)$ of modules with a locally nilpotent action of $\u_+$. 
A basic example of an object of $HC_+(\g,K)$ is the parabolic Verma module 
$M_+(X)=U(\g)\otimes_{U(\kk\oplus \u_+)}X$, where $X\in \Rep(K)$ is a simple object and $\u_+$ acts on $X$ by zero. 
In this case we have the Harish-Chandra homomorphism $HC : Z(\g)\to Z(\kk)$ defined as in the classical case
(namely, by the condition that $z\in HC(z)+U(\g)\u_+$), and the central character of $M_+(X)$ is defined by $\chi_X(z):=HC(z)|_X$. 

The objects $M_+(X)$ are interpolations of holomorphic discrete series modules 
in the classical case. It would be interesting to study the reducibility of $M_+(X)$.

The category $HC_+(\g,K)$ is a subcategory of the appropriate parabolic category O, discussed in the next section. 

\section{Parabolic category O}

We would now like to extend the theory of category O to the setting of Deligne categories. Unfortunately, it is not clear how to define category O in this setting, 
since for $\g={\mathfrak{gl}}_t, {\mathfrak{o}}_t, {\mathfrak{sp}}_{2t}$ it is not clear what a Borel subalgebra is.  
However, one can define the parabolic category O attached to a parabolic subalgebra.
Before doing so, let us review examples of parabolic subalgebras that can be defined 
in the setting of Deligne categories. 

\begin{example}\label{glpar} Let $G=GL_t$, $\g={\mathfrak{gl}}_t$, and $t=t_1+...+t_m$, with $t_i,t\notin \Bbb Z$. Then we have a forgetful functor $\Rep(G)\to \Rep(L)$, where 
$L=GL_{t_1}\times...\times GL_{t_m}$ is a "Levi subgroup" 
(and by $\Rep(L)$ we mean $\Rep(GL_{t_1})\boxtimes...\boxtimes \Rep(GL_{t_m})$). Let $V_i$ be the tautological objects of $\Rep(GL_{t_i})$. 
Then in $\Rep(L)$ we have a decomposition $\g=\u_+\oplus \l\oplus \u_-$, where $\u_+=\oplus_{i<j}V_i\otimes V_j^*$, 
$\u_-=\oplus_{i<j}V_i^*\otimes V_j$, and $\l=\oplus_i V_i\otimes V_i^*={\rm Lie}L$. So we have the parabolic subalgebra 
$\p_+=\l\oplus \u_+$, with Levi subalgebra $\l$ and unipotent radical $\u_+$ (in $\Rep(L)$). 

Note that we have a modification of this example, where a subset of the numbers $t_i$, e.g. $t_1,...,t_s$, are positive integers, 
and we use the classical representation categories ${\bf Rep}(GL_{t_i})$ instead of $\Rep(GL_{t_i})$ for $1\le i\le s$. 
\end{example} 

\begin{example}\label{osppar} Now let $G={\mathfrak{o}}_{2t}$. Given a decomposition $t=t_0+t_1+...+t_m$ with $t_i, i\ge 1, 2t_0,2t\notin \Bbb Z$, 
we have a Levi subgroup $L=O_{2t_0}\times GL_{t_1}\times...\times GL_{t_m}$. Let $V_0,V_1,...,V_m$ be the corresponding tautological objects. 
Then we have $\l={\rm Lie}L$, $\u_+=\oplus_{i<j}V_i\otimes V_j^*$, $\u_-=\oplus_{i<j}V_i^*\otimes V_j$, and $\g=\u_+\oplus \l\oplus \u_-$.  

Also, as in Example \ref{glpar}, we may fix a subset of the numbers $2t_0,t_1,...,t_m$ 
to be positive integers, and use the classical representation categories instead of Deligne categories.
This includes the case when $t_0=0$, when we don't have a factor $O_{2t_0}$.

The same definition can be made for $G={\mathfrak{sp}}_{2t}$, with $L=Sp_{2t_0}\times GL_{t_1}\times...\times GL_{t_m}$.

In all of these cases, we have a parabolic subalgebra 
$\p_+=\l\oplus \u_+$, with Levi subalgebra $\l$ and unipotent radical $\u_+$ (in $\Rep(L)$). 
\end{example} 

Now we can define the parabolic category O. Let $\l=\z\oplus [\l,\l]$, where $\z$ is the center of $\l$. 

\begin{definition} The category $O(\g,L,\u_+)$, called the parabolic category O, is the category of finitely generated $U(\g)$-modules $M$ in ${\rm Ind}\Rep(L)$ such that 

(i) the action of $\u_+$ on $M$ is locally nilpotent; 

(ii) the action of $[\l,\l]$ on $M$ coincides with its natural action (via the embedding $[\l,\l]\subset \l$), 
and the action of $\z$ on $M$ is semisimple. 
\end{definition} 

Typical objects in $O(\g,L,\u_+)$ are parabolic Verma modules. Namely, fix a weight $\lambda\in \z^*$ and a simple object $X\in \Rep(L)$. 
Let $\bold 1_\lambda$ be the $\l$-module which is $\bold 1$ as an object of $\Rep(L)$, and such that $\z$ acts via $\lambda$, while $[\l,\l]$ acts trivially. 

\begin{definition} 
The parabolic Verma module $M_+(X,\lambda)$ is defined by the formula 
$M_+(X,\lambda)=U(\g)\otimes_{U(\l\oplus \u_+)}(X\otimes \bold 1_\lambda)$, where 
 $\u_+$ acts on $X\otimes \bold 1_\lambda$ by zero.  
 We will mostly use the abbreviated notation $M(X,\lambda)$. 
\end{definition} 

Thus, as an ind-object of $\Rep(L)$, we have $M(X,\lambda)=U(\u_-)\otimes X$. 

Let $z_1,...,z_m$ be the natural basis of $\z$ (namely, $z_i$ is the unit of ${\mathfrak{gl}}_{t_i}$). 
Then the Lie algebra $\g$ in $\Rep(L)$ has a $\Bbb Z^m$-grading by eigenvalues of 
the adjoint action of $z_1,...,z_m$, and the eigenobjects are finite dimensional 
(i.e., are finite length objects of $\Rep(L)$). Thus, for every module $M\in O(\g,L,\u_+)$ 
we can define its character with values in the Grothendieck ring of $\Rep(L)$:  
$$
{\rm ch}M(x_1,...,x_m)=\sum_{a_1,...,a_m} M[a_1,...,a_m]x_1^{a_1}...x_m^{a_m},
$$
where $M[a_1,...,a_m]$ is the common eigenobject of $z_1,...,z_m$ on $M$ 
with eigenvalues $a_1,...,a_m$. 

Let $M_-(X,\lambda)$ be the module defined in the same way as 
$M_+(X,\lambda)$, replacing $\u_+$ with $\u_-$. As objects we have $M_-(X,\lambda)=U(\n_+)\otimes X$. 

\begin{proposition}\label{stanO} 
(i) The module $M_+(X,\lambda)=M(X,\lambda)$ has a maximal proper submodule $J(X,\lambda)$ and a unique irreducible quotient $L(X,\lambda)=M(X,\lambda)/J(X,\lambda)$.  

(ii) Every irreducible object in $O(\g,L,\u_+)$ is of the form $L(X,\lambda)$ for a unique $X,\lambda$.

(iii) There is a unique $\g$-invariant pairing 
$$
(,)_\lambda: M_+(X,\lambda)\otimes M_-(X^*,-\lambda)\to \bold 1
$$ 
(the Shapovalov form) which coincides with the evaluation morphism on $X\otimes X^*$. 
The left kernel of $(,)_\lambda$ is $J(X,\lambda)$. 

(iv) For generic $\lambda$ (i.e., outside of countably many hypersurfaces), $M(X,\lambda)$ is irreducible (i.e., $M(X,\lambda)=L(X,\lambda)$).
\end{proposition}

\begin{proof}
The proof is standard. 

Namely, (i) is proved using that every submodule of $M(X,\lambda)$ is graded by $\Bbb Z^m$, and 
thus the sum of all proper submodules of $M(X,\lambda)$ is a proper submodule. 

(ii) follows from the fact that any simple object of $O(\g,L,\u_+)$ contains a subobject annihilated by $\u_+$ 
(by the local nilpotence of the action of $\u_+$). 

(iii) Follows since we can view the pairing in question as a homomorphism $M_+(X,\lambda)\to M_-(X^*,-\lambda)^*$, and 
it's easy to see that such a homomorphism acting by the identity on the highest component exists and is unique, and its image is $L(X,\lambda)$. 

Finally, (iv) follows (in the same way as in the classical case, using (iii)) from the fact that for generic $\lambda$, the pairing $B: \u_+\otimes \u_-\to \bold 1$ 
defined by $B:=\lambda\circ [,]$ is nondegenerate. 
\end{proof} 
 
\begin{remark}
In fact, one can determine the set of $\lambda$ for which (iv) fails by interpolating the Jantzen determinant formula 
for the Shapovalov form on parabolic Verma modules (\cite{Ja}), to get a formula for the determinant of $(,)_\lambda$ on isotypic components. 
\end{remark}  

This gives rise to the following problem. 

\begin{problem} Compute the characters of $M(X,\lambda)$ for all $X,\lambda,t,t_i$
\end{problem} 

This should, in particular, lead to some stable limits of Kazhdan-Lusztig polynomials. 

\section{Lie supergroups}
 
Representations categories  of classical Lie supergroups are interpolated to complex rank quite similarly to categories of $(\g,K)$-modules;
namely, as before, we have a decomposition $\g=\kk\oplus \p$, and the only difference is that $\p$ is odd. 
Indeed, let us define the representation categories 
of the classical Lie supergroups $GL_{t|s}$ and $OSp_{t|2s}$. 

To define $G=GL_{t|s}$ for $t,s\notin \Bbb Z$, let $K=GL_t\times GL_s$. 
Namely, let $V,U$ be the tautological objects of $\Rep(GL_t)$ and $\Rep(GL_s)$, respectively. Then we have a Lie superalgebra 
$\g={\mathfrak{gl}}_{t|s}=\kk\oplus \p$ in $\Rep(K)$, where $\kk={\rm Lie}K$ and $\p=V\otimes U^*\oplus U\otimes V^*$, with the supercommutator 
$S^2\p\to \kk$ defined in the obvious way (pairing the two summands in $\p$). 

Similarly, to define $G=OSp_{t|2s}$ for $t,2s\notin \Bbb Z$, let $K=O_t\times Sp_{2s}$. 
Let $V,U$ be the tautological objects in $\Rep(O_t)$, $\Rep(Sp_{2s})$. Then we can define $\g={\mathfrak{osp}}_{t|2s}=\kk\oplus \p$, where $\kk={\rm Lie}K$, and 
$\p=V\otimes U$, with the obvious supercommutator $S^2\p\to \kk$. 

\begin{definition} In both cases, the category $\Rep(G)$ is the category of $\g$-modules in ${\rm Ind}\Rep(K)$, 
such that the action of $\kk\subset \g$ is the natural action. 
\end{definition} 

\begin{remark} In the classical case, the algebra $\wedge \p$ is finite dimensional, 
and hence any $G$-module is locally finite (i.e., a sum of finite dimensional modules). 
However, this is no longer the case in the Deligne category setting, which is why we are considering 
$\g$-modules in ${\rm Ind}\Rep(K)$ rather than $\Rep(K)$. 
\end{remark} 

We can also define this category for $t$ or $s$ being a positive integer, using the usual representation category of the corresponding group
instead of the Deligne category. 

\begin{remark}
Note that any object $Y\in \Rep(G)$ has a natural $\Bbb Z_2$-grading, by the number of $U$-factors mod 2. 
\end{remark} 

An interesting question is to compute the structure of irreducible representations of $G$. 
For instance, for $G=GL_{t|s}$, we can look at irreducible representations 
with locally nilpotent action of $V\otimes U^*$ (i.e., those which lie in a suitable parabolic category O).
Any such representation $L$ has a unique simple $\Rep(K)$-subobject $X$ killed by $V\otimes U^*$, which determines $L$; we write $L=L(X)$.  
The representations $L(X)$ are $\Bbb Z$-graded (by the number of $V$-factors) with finite dimensional homogeneous 
components, so one may raise the following problem. 

\begin{problem} Compute the character of $L(X)$ for each $X$, i.e., the Hilbert series of $\Hom(Y,L(X))$ for all $Y\in \Rep(K)$. 
\end{problem} 

In the classical case $(t,s\in \Bbb Z)$, this problem was solved in \cite{Se}, 
in terms of a particular kind of Kazhdan-Lusztig polynomials. It would be interesting 
to interpolate this result to complex values of $t$ and $s$. 

\section{Affine Lie algebras} 

Let $\g$ be a Lie algebra in a symmetric tensor category ${\mathcal{C}}$. In this case, given any commutative algebra $R$ in ${\mathcal{C}}$
(for example, an ordinary algebra over $\Bbb C$), we can form a new Lie algebra $\g\otimes R$. In particular, if $R=\Bbb C[z,z^{-1}]$, we obtain 
the loop algebra $L\g=\g[z,z^{-1}]$. 

Now let $\g$ be a quadratic Lie algebra, i.e., a Lie algebra with a symmetric nondegenerate inner product $B: \g\otimes \g\to \bold 1$
(in other words, we have a symmetric isomorphism of $\g$-modules $\g\cong \g^*$). In this case, one can define a 1-dimensional central extension 
$\widehat{\g}$ of $L\g$, using the 2-cocycle $\omega: L\g\otimes L\g\to \bold 1$, given by 
$\omega|_{\g z^m\otimes \g z^n}=m\delta_{m,-n}B$. The Lie algebra $\widehat{\g}$ is called the {\it affine Lie algebra} attached to $\g$. 
Moreover, we have an action of the Virasoro algebra ${\rm Vir}$ on $\widehat{\g}$ by $L_n|_{\g z^m}=-mz^n: \g z^m\to \g z^{m+n}$, and we can form the semidirect
product ${\rm Vir}\ltimes \widehat{\g}$. 

In this setting, we can generalize some standard results about affine Lie algebras. For instance, we have the Sugawara construction. 
Namely, for any Lie algebra $\g$ in ${\mathcal{C}}$ one can define the (symmetric) Killing form ${\rm Kil}: \g\otimes \g\to \bold 1$ by the formula 
$$
{\rm Kil}={\rm ev}_{\g^*}\circ ([,]\otimes {\rm Id}_{\g^*})\circ ({\rm Id}_\g\otimes [,]\otimes {\rm Id}_{\g^*})\circ {\rm Id}_{\g\otimes\g}\otimes {\rm coev}_\g
$$

Now for a quadratic $\g$, let us call a number $k\in \Bbb C$ {\it non-critical} if the form $B_k:=kB+\frac{1}{2}{\rm Kil}: \g\otimes \g\to \bold 1$ 
is nondegenerate. Also, for any $i,j\in \Bbb Z$ define a morphism $C_{ij}(k): \bold 1\to U(\widehat{\g})$ by the formula 
$$
C_{ij}(k)={\rm mult}(B_k^{-1}\cdot(z^i\otimes z^j))\text{ if }i\le j,\ C_{ij}(k)=C_{ji}(k). 
$$

Let us say that a $\widehat{\g}$-module $M$ in ${\rm Ind}{\mathcal{C}}$ is {\it of level $k$} if the central subobject $\bold 1$ of $\widehat{\g}$ acts by $k$ in $M$.
Also, let us say that $M$ is {\it admissible} if for every finite length ${\mathcal{C}}$-subobject $X\subset M$ 
there exists $N\in \Bbb N$ such that for all $n\ge N$, the action map $\g z^n\otimes X\to M$ is zero. 

\begin{proposition}\label{sugawa} Let $M$ be an admissible ${\widehat{\g}}$-module in ${\mathcal{C}}$ of non-critical level $k$. 
Then the action $\widehat{\g}$ on $M$ extends to an action of ${\rm Vir}\ltimes \widehat{\g}$
via the Sugawara formula: 
$$
L_n=\frac{1}{2}\sum_{i+j=n}C_{ij}(k).
$$
Moreover, the Virasoro central charge of this action equals 
$$
c=kB\circ B_k^{-1}.
$$ 
\end{proposition} 

\begin{proof} The proof is standard, see e.g. \cite{Ka}. 
\end{proof} 

In particular, if $\g$ is a simple Lie algebra (i.e., $\g$ is a simple $\g$-module), then ${\rm Kil}=gB$, where $g$ is the "dual Coxeter number" of $\g$ 
(with respect to $B$). Then we get that $k$ is non-critical iff $k\ne -g$ (so $-g$ is called the critical level), and $c=\frac{k \dim\g}{k+g}$. 

Let us now specialize to the case when ${\mathcal{C}}=\Rep(G)$, where $G=GL_t$, $O_t$, or $Sp_{2t}$, and $\g={\rm Lie}(G)$, 
with the form $B$ being the interpolation of the trace form (in the first case, we will also consider $\g_0={\mathfrak{sl}}_t$, 
which, unlike ${\mathfrak{gl}}_t$, is a simple Lie algebra). Then the Sugawara construction applies, 
with the dual Coxeter numbers $g=t$ for ${\mathfrak{sl}}_t$, $g=t-2$ for ${\mathfrak{o}}_t$, and $g=2t+2$ for 
${\mathfrak{sp}}_{2t}$ (note that the dual Coxeter number of ${\mathfrak{sp}}_{2n}$ is $n+1$, but we use a different normalization 
of the bilinear form from the standard one, which gives twice as much). 

One can also consider the theory of parabolic category O for $\widehat{\g}$, similarly to Section 4. Namely, we define the category 
$O_k(\widehat{\g})$ to be the category of $\widehat{\g}$-modules of level $k$ in $\Rep(G)$ on which the action of $\g$ 
is the natural one, and the action of $z\g[z]$ is locally nilpotent. Typical objects 
of $O_k(\widehat{\g})$ are Verma modules 
$$
M(X,k)=U(\widehat{\g})\otimes_{U(\g[z]\oplus \bold 1)}X=U(z^{-1}\g [z^{-1}])\otimes X,\ X\in \Rep(G), 
$$
where 
$\bold 1$ acts on $X$ by $k$ and $z\g [z]$ by zero. Similarly to Section 4, this module 
admits a Shapovalov form and has a unique simple quotient $L(X,k)$, and $M(X,k)=L(X,k)$ for all but countably many $k$.   
In fact, by looking at the action of the Casimir operator $L_0+d$ (where $d$ is the degree operator), 
one can check that the numbers $k$ for which $M(X,k)\ne L(X,k)$ are all of the form $r_1t+r_2$, where $r_1,r_2\in \Bbb Q$. 

This gives rise to the following problem: 

\begin{problem} Determine the characters of $L(X,k)$ in the case when $L(X,k)\ne M(X,k)$. 
\end{problem} 

As an example, consider the basic representation of $\widehat{\g}_0$, where $G=GL_t$ and $\g_0={\mathfrak{sl}}_t$, namely, $\bold V:=L(\bold 1,k=1)$. 
Note that this representation is graded by powers of $z$. Thus, for any $X\in \Rep(G)$, we can ask for the Hilbert series of the isotypic component of $X$, 
$$
h_X(q)=\sum_n \dim\Hom(X,\bold V)[-n]q^n.
$$
Let us determine $h_X(q)$. To do so, recall that in the classical case of ${\mathfrak{sl}}_n$, we have the Frenkel-Kac vertex operator construction (\cite{Ka}), which 
gives the character formula 
$$
{\rm ch}\bold V=\frac{\sum_{\beta\in Q}q^{\beta^2/2}e^\beta}{\prod_{j\ge 1}(1-q^j)^{n-1}},
$$
where $Q$ is the root lattice. 
So we would like to interpolate this formula. For this purpose we'd like to write this sum as a linear combination of irreducible characters $\chi_\lambda$  of ${\mathfrak{sl}}_n$. 
This formula is well known (see \cite{Ka}, Exercise 12.17), but we recall its derivation for reader's convenience. We have 
$$
{\rm ch}\bold V=\sum_{\lambda \in Q\cap P_+}C_{\lambda,n}(q)\chi_\lambda,
$$
where 
$$
C_{\lambda,n}(q)=|W|^{-1}\frac{\sum_{\beta\in Q}q^{\beta^2/2}(\Delta^2e^\beta,\chi_\lambda)}{\prod_{j\ge 1}(1-q^j)^{n-1}},
$$
$P_+$ is the set of dominant integral weights, 
$(,)$ denotes the inner product defined by $(e^\beta,e^{\gamma})=\delta_{\beta,-\gamma}$, $W=S_n$ is the Weyl group, and $\Delta$ is the Weyl denominator. 
By the Weyl character formula, we have 
$$
\Delta \chi_\lambda=m_{\lambda+\rho}^-:=\sum_{w\in W}(-1)^w e^{w(\lambda+\rho)}.
$$
 Thus we get 
$$
C_{\lambda,n}(q)=|W|^{-1}\frac{\sum_{\beta\in Q}q^{\beta^2/2}(\Delta e^\beta, m_{\lambda+\rho}^-)}{\prod_{j\ge 1}(1-q^j)^{n-1}}.
$$
Now by the Weyl denominator formula, $\Delta=\sum_{w\in W}(-1)^w e^{w\rho}$, so we get 
$$
C_{\lambda,n}(q)=\frac{\sum_{w\in W}(-1)^wq^{(\lambda+\rho-w\rho)^2/2}}{\prod_{j\ge 1}(1-q^j)^{n-1}}=
q^{\lambda^2/2}\frac{\prod_{\alpha>0}(1-q^{(\lambda+\rho,\alpha)})}{\prod_{j\ge 1}(1-q^j)^{n-1}}. 
$$
Now we can see that if $\lambda$ and $\mu$ are partitions with $|\lambda|=|\mu|$, then 
$C_{[\lambda,\mu]_n,n}(q)$ has a limit $C_{\lambda,\mu,\infty}(q)$ as $n\to \infty$.
For example, if $\lambda=\mu=0$, we get 
$$
C_{0,0,\infty}(q)=\prod_{j=2}^\infty (1-q^j)^{-j+1}.
$$
In general, if $\lambda=(\lambda_1,...,\lambda_r)$ and $\mu=(\mu_1,...,\mu_s)$, we get 
$$
C_{\lambda,\mu,\infty}(q)=
q^{\frac{\lambda^2+\mu^2}{2}}C_{0,0,\infty}(q)\prod_{1\le i<j\le r}\frac{1-q^{\lambda_i-\lambda_j+j-i}}{1-q^{j-i}}\prod_{1\le i<j\le s}\frac{1-q^{\mu_i-\mu_j+j-i}}{1-q^{j-i}}\times
$$
$$
\prod_{i=1}^r\prod_{j=0}^{\lambda_i-1}(1-q^{r+1+j-i})^{-1}
\prod_{i=1}^s\prod_{j=0}^{\mu_i-1}(1-q^{s+1+j-i})^{-1}
$$
For example, if $\lambda=\mu=(p)$, we get 
$$
C_{p,p,\infty}(q)=q^{p^2}C_{0,0,\infty}(q)\prod_{j=1}^p (1-q^j)^{-2}=
$$
$$
q^{p^2}(1-q)^{-2}(1-q^2)^{-3}...(1-q^p)^{-p-1}(1-q^{p+1})^{-p}(1-q^{p+2})^{-p-1}...
$$

Thus, we obtain the following proposition. 

\begin{proposition}\label{charform}
The Hilbert series of $\Hom(X_{\lambda,\mu},\bold V)$ equals $C_{\lambda,\mu,\infty}(q)$.  
\end{proposition} 

Similarly, if $\widetilde{\bold V}$ is the basic representation of $\widehat{\g}$, where 
$\g={\mathfrak{gl}}_t$, then $\widetilde{\bold V}={\bold V}\otimes {\mathcal{F}}$, where ${\mathcal{F}}$ is the standard Fock space, so the Hilbert series of $\Hom(X_{\lambda,\mu},\widetilde{\bold V})$ 
is $\widetilde{C}_{\lambda,\mu,\infty}(q)$, where 
$\widetilde{C}_{\lambda,\mu,\infty}(q)=C_{\lambda,\mu,\infty}(q)\prod_{i=1}^\infty (1-q^i)^{-1}.$
For example, 
$\widetilde{C}_{0,0,\infty}(q)=\prod_{j=1}^\infty (1-q^j)^{-j}.$

\begin{remark} Note that in the classical situation, $\widetilde{\bold V}$ is a vertex operator algebra, and $\widetilde{\bold V}^G$ 
is known to be the affine $W$-algebra $W_n=W({\mathfrak{gl}}_n)$ of central charge $n$ (\cite{F}; see also \cite{FKRW}). Similarly, 
in the Deligne category setting, $\widetilde{\bold V}$ is a vertex operator algebra in ${\rm Ind}\Rep(G)$, 
and $\widetilde{\bold V}^G$ is the $W_{1+\infty}$ vertex operator algebra with central charge $t$, see \cite{FKRW}. 
Moreover, the spaces $\Hom(X_{\lambda,\mu},\widetilde{\bold V})$ are modules over this vertex operator algebra.  
\end{remark}

\section{Yangians} 

As in the previous section, let $\g$ be a quadratic Lie algebra in a symmetric tensor category ${\mathcal{C}}$. 
In this case, following Drinfeld (\cite{Dr}; see \cite{CP} for a detailed discussion), one can define an algebra $Y(\g)$ in ${\mathcal{C}}$ called the {\it Yangian} of $\g$, which 
is a Hopf algebra deformation of $U(\g[z])$. More precisely, Drinfeld gave a definition of $Y(\g)$ when $\g$ is a simple Lie algebra in the category of vector spaces, but the definition extends verbatim to our more general setting. In particular, this construction allows us to define $Y(\g)$ for $\g={\rm Lie}G$, where $G=GL_t$, $O_t$, or $Sp_{2t}$. 
This $Y(\g)$ is a Hopf algebra in ${\rm Ind}\Rep(G)$. 

The PBW theorem for $Y(\g)$ says that $Y(\g)$ has a filtration such that ${\rm gr}Y(\g)=U(\g[z])$ (with grading by powers of $z$); it particular, it contains $U(\g)$ as a Hopf subalgebra. 
As usual, it is not hard to see that there is a surjective map $U(\g[z])\to {\rm gr}Y(\g)$ 
(as the defining relations of the Yangian deform the defining relations of $U(\g[z])$), but it 
is harder to show that this map is also injective. This result was proved by Drinfeld in the classical setting, 
and therefore follows in Deligne categories by interpolation.

Drinfeld's relations for $Y(\g)$ are rather complicated, so let us give a different, simpler presentation of $Y(\g)$ for $G=GL_t$, which is 
an interpolation of the Faddeev-Reshetikhin-Takhtajan presentation (see \cite{Mo} for a review and references). To introduce this presentation, let $V$ be the tautological object of $\Rep(G)$, 
and start with the tensor algebra $A:={\bold T}(\oplus_{i\ge 0} (V\otimes V^*)_{(i)})$. Let $T_i\in \Hom(\bold 1,(V^*\otimes V)\otimes A)$ 
be the coevaluation map 
$$
\bold 1\to (V^*\otimes V)\otimes (V\otimes V^*)_{(i)}.
$$
Let $T(u)=1+T_0u^{-1}+T_1u^{-2}+...$ 
be the generating function of $T_i$, and let $R(u)=1+\frac{\sigma}{u}$, where $\sigma: V\otimes V\to V\otimes V$ is the permutation. 
Consider the series 
$$
Q(u,v):=(u-v)(R^{12}(u-v)T^{13}(u)T^{23}(v)-T^{23}(v)T^{13}(u)R^{12}(u-v))
$$
 in $u^{-1}$ and $v^{-1}$ with coefficients $Q_{ij}\in \Hom(\bold 1, (V^*\otimes V)\otimes (V^*\otimes V)\otimes A)$. 
(so that $Q=\sum Q_{ij}u^iv^j$). Regard $Q_{ij}$ as a morphism 
$$
Q_{ij}:=(V\otimes V^*)\otimes (V\otimes V^*)\to A
$$
(landing in degree $2$). Let $J$ be the ideal in $A$ generated by the images of all the $Q_{ij}$. 

\begin{definition}\label{yanggln} The algebra $Y(\g):=A/J$ is called the Yangian of $\g$. 
\end{definition}  

One can show that Definition \ref{yanggln} is equivalent to the above (it is shown in the same way as in the classical case). 
In particular, the copy of $U(\g)$ inside $Y(\g)$ is generated by the image of $T_0$ (regarded as a morphism $V\otimes V^*\to Y(\g)$); 
moreover, $T_i$ corresponds in the associated graded algebra to $\g z^i$. Finally, the Hopf algebra structure 
is written very simply in terms of this presentation: $\Delta(T(u))=T^{12}(u)T^{13}(u)$, $\varepsilon(T(u))=1$, $S(T(u))=T(u)^{-1}$
(where $\Delta$ is the coproduct, $\varepsilon$ the counit, and $S$ is the antipode). 

Besides greater simplicity than the general definition, Definition \ref{yanggln} has the important advantage that 
it comes with a family of representations. Namely, since $R$ satisfies the quantum Yang-Baxter equation 
$$
R^{12}(u_1-u_2)R^{13}(u_1-u_3)R^{23}(u_2-u_3)=R^{23}(u_2-u_3)R^{13}(u_1-u_3)R^{12}(u_1-u_2),
$$
we find that the assignment 
$T(u)\mapsto R(u-z)$, $z\in \Bbb C$, defines a homomorphism ${\rm ev}_z: Y(\g)\to U(\g)$, 
called the {\it evaluation homomorphism}. For $X\in \Rep(G)$, denote by $X(z)$ the pullback ${\rm ev}_z^*X$ to a representation of $Y(\g)$.
Then, given any simple objects $X_1,...,X_k\in \Rep(G)$ and $z_1,...,z_k\in \Bbb C$, we can construct the representation 
$X_1(z_1)\otimes....\otimes X_k(z_k)$ of $Y(\g)$. As in the classical case, these representations are irreducible for generic 
parameter values, but become reducible for special values, and other irreducible representations 
are obtained as their composition factors. 

This gives rise to the following problem.

\begin{problem} Classify irreducible representations of $Y(\g)$ in $\Rep(G)$ on which the action of $\g$ is the natural one
(generalizing the theory of Drinfeld polynomials, \cite{CP}) and compute their decompositions into simple objects of $\Rep(G)$.
\end{problem} 
 
It is known that in the classical setting, these irreducible representations have a rich structure, related to 
geometry of quiver varieties, cluster algebras, Hirota bilinear relations, etc.

\end{document}